\documentclass[oneside,11pt,a4paper,reqno,ps]{article}
\usepackage[a4paper,top=2.5cm,bottom=2.5cm,left=2.5cm,right=2.5cm,
bindingoffset=5mm]{geometry}
\usepackage[latin1]{inputenc}
\usepackage[english]{babel}
\usepackage[T1]{fontenc}
\usepackage{syntonly}
\usepackage{graphicx}
\usepackage{amssymb}	
\usepackage{lmodern}
\usepackage{amsthm}
\usepackage{microtype}
\usepackage{amsmath}
\usepackage{tikz}
\usepackage{braket}
\usepackage{eurosym}
\usepackage{amssymb}
\usepackage{amstext}
\usepackage{color}
\usepackage{xcolor}
\usepackage{url}
\usepackage{amsmath,amsthm,amssymb,amscd,epsfig,esint, mathrsfs,graphics,color}
\usepackage{wrapfig}
\usepackage{verbatim}
\usepackage{hyperref}
\usepackage{listings}
\usepackage{bookmark}
\usepackage{csquotes}
\hypersetup{%
	pdfpagemode={UseOutlines},
	bookmarksopen,
	pdfstartview={FitH},
}
\usepackage[maxcitenames=99, maxbibnames=99, backend=bibtex, hyperref]{biblatex}
\newcommand{\numberset}{\mathbb}
\newcommand{\R}{\numberset{R}}

\theoremstyle{plain}
\newtheorem{thm}{Theorem}[section]

\newtheorem{proposition}[thm]{Proposition}
\newtheorem{lemma}[thm]{Lemma}
\newtheorem{corollary}[thm]{Corollary}

\theoremstyle{definition}
\newtheorem{definition}[thm]{Definition}

\def\XXint#1#2#3{{\setbox0=\hbox{$#1{#2#3}{\int}$} 
		\vcenter{\hbox{$#2#3$}}\kern-.5\wd0}}

\def\eps{\varepsilon}

\def\div{{\rm div}}

\def\supp{\mathrm{supp}}
\def\loc{\mathrm{loc}}
\numberwithin{equation}{section} \makeatletter

\renewcommand{\p@enumi}{\thesection.}
\title{\textbf{Higher regularity for weak solutions\\
		to degenerate parabolic problems}}
\author{\sc{Andrea Gentile - Antonia Passarelli di Napoli\footnote{{\bf Aknowledgments.} The work of the authors is supported by GNAMPA (Gruppo Nazionale
			per l'Analisi Matematica, la Probabilit\`a e le loro Applicazioni) of INdAM (Istituto
			Nazionale di Alta Matematica).
			The authors have been also supported  by the Universit\'a degli Studi di Napoli    ``Federico II'' through the project FRA-000022-ALTRI-CDA-752021-FRA-PASSARELLI.}}\\ \\
	\small{Dipartimento di Matematica e Applicazioni ``R.
		Caccioppoli''} \\ \small{Universit\'{a} di Napoli ``Federico
		II'', via Cintia - 80126 Napoli}
	\\ \small{e-mail: \href{mailto:andrea.gentile@unina.it}{andrea.gentile@unina.it},\href{mailto:antpassa@unina.it}{antpassa@unina.it}}} 

\begin{document}
	\maketitle
	
	\begin{abstract}
		\noindent In this paper, we study the regularity of weak solutions to the following strongly degenerate
		parabolic equation
		\begin{equation*}
			u_t-\div\left(\left(\left|Du\right|-1\right)_+^{p-1}\frac{Du}{\left|Du\right|}\right)=f\qquad\mbox{ in }\Omega_T,
		\end{equation*}
		where $\Omega$ is a bounded domain in $\mathbb{R}^{n}$ for $n\geq2$,
		$p\geq2$ and $\left(\,\cdot\,\right)_{+}$ stands for the positive
		part. We prove  the higher differentiability of a nonlinear function of the spatial gradient of the weak solutions,
		assuming only that $f\in L^{2}_{\loc}\left(\Omega_T\right)$. This allows us to establish the higher integrability of the spatial gradient under the same minimal requirement on the datum $f$.
	\end{abstract}
	
	\bigskip
	
	\hspace{1cm}\small{{\bf Key words.} Widely degenerate problems. Second order regularity. Higher integrability.}
	
	\bigskip
	
	\hspace{1cm}\small{{\bf AMS Classification.} 35B45, 35B65, 35D30, 35K10, 35K65}
	
	\section{Introduction}
	{In this paper, we study the regularity properties of weak solutions $u:\Omega_{T}\rightarrow\mathbb{R}$ to the following parabolic equation 
		\begin{equation}\label{eq1*}
			u_{t}-\mathrm{div}\left(\left(\vert Du\vert-1\right)_{+}^{p-1}\frac{Du}{\vert Du\vert}\right)=f\qquad\mbox{ in }\Omega_{T}=\Omega\times(0,T),
		\end{equation}
		which appears in gas filtration problems
		taking into account the initial pressure gradient. For a precise description of this motivation we refer to \cite{Baremblatt} and \cite[Section 1.1]{AmbrosioPdN}.
		\\
		The main feature of   this equation is that it
		possesses a wide degeneracy, coming from the fact that its modulus of ellipticity vanishes at all points where $\left|Du\right|\le 1$ and hence 
		its principal part of  behaves like a $p$-Laplacian operator only at infinity.
		\\
		In this paper we address two interrelated aspects of the regularity theory for solutions to parabolic problems, namely the higher differentiability and the higher integrability of the weak solutions to \eqref{eq1*}, with the main aim of weakening the assumption on the datum $f$ with respect to the available literature.
		\\
		These questions have been exploited in case of non degenerate parabolic problems with quadratic growth by Campanato in \cite{Campanato}, by Duzaar et al. in \cite{DMS} in case of superquadratic growth, while Scheven in \cite{Sche} faced the subquadratic growth case. In the above mentioned papers, the problem have been faced or in case of homogeneous  equations or considering sufficiently regular datum. It is worth mentioning that  the higher integrability of the gradient of the solution is achieved through an interpolation argument, once its higher differentiability is established.
	}

	
	
	
	{This strategy has revealed to be successful  also for degenerate equations as in \eqref{eq1*}. Indeed the higher integrability
		of the spatial gradient of weak solutions to equation \eqref{eq1*},
		has been proven in \cite{AmbrosioPdN} , under  suitable
		assumptions on the datum $f$ in the scale of Sobolev spaces.}
	
	{We'd like to recall that a common feature for nonlinear problems with growth rate $p>2$ is  that the higher differentiability is proven for a nonlinear expression of the gradient which takes into account the growth of the principal part of the equation.\\
		Indeed, already for the non degenerate $p$-Laplace equation, the higher differentiability refers to the function $V_p\left(Du\right)=\left(1+\left|Du\right|^2\right)^\frac{p-2}{4}Du$.
		In case of widely degenerate problems, this phenomenon persists, and higher differentiability results, both for the  elliptic and the  parabolic problems, hold true for the function $H_\frac{p}{2}\left(Du\right)=\left(\left|Du\right|-1\right)_+^{\frac{p}{2}}\displaystyle\frac{Du}{\left|Du\right|}$.
			It is worth noticing that, as it can be expected, this function of the gradient doesn't give  information on the second  regularity of the solutions in the set where the equation degenerates.
			Actually, since every 1-Lipschitz continuous function is a solution to the elliptic equation
			$$\div\left(H_{p-1}\left(Du\right)\right)=0,$$
			{where $H_{p-1}\left(Du\right)=\left(\left|Du\right|-1\right)_+^{p-1}\displaystyle\frac{Du}{\left|Du\right|}$}, no more than Lipschitz regularity
			can be expected.
			\\
			Moreover, it is well known that in case of degenerate problems (already for the degenerate $p$-Laplace equation, with $p>2$) a Sobolev regularity is required for the datum $f$ in order to get the higher differentiability of the solutions {(see, for example \cite{Br1} for elliptic and \cite{AmbrosioPdN} for parabolic equations)}. Actually, the sharp assumption for the datum in the elliptic setting has been determined in \cite{Br1} as a fractional Sobolev regularity suitably  related to the growth exponent $p$ and the dimension $n$.\\
			The main aim of this paper is to show that without  assuming any kind of Sobolev regularity for the datum, but assuming only $f\in L^2$, we are still able to obtain higher differentiability for the weak solutions but outside a set larger than the degeneracy set of the problem. {It is worth mentioning that, while for the $p$-Laplace equation the degeneracy appears for $p>2$, here, even in case $p=2$, under a $L^2$ integrability assumption on the datum $f$, the local $W^{2,2}$ regularity of the solutions cannot be obtained.}
			\\
			{Actually, we shall prove the following}}
		\begin{thm}\label{Theorem1}
			Let $n\ge2$, $p\ge2$ and $f\in L^{2}_{\loc}\left(\Omega_T\right)$. Moreover, let us assume that 
			$$
			u\in C^0\left(0, T; L^2\left(\Omega\right)\right)\cap L^p_\loc\left(0, T; W^{1, p}_\loc\left(\Omega\right)\right)
			$$
			is a weak solution to \eqref{eq1*}. Then, for any $\delta\in(0,1)$, we have
			$$
			\mathcal{G}_\delta\left(\left(\left|Du\right|-1-\delta\right)_+\right)\in L^2_\loc\left(0,T; W^{1,2}_\loc\left(\Omega\right)\right),
			$$
			where
			$$
			\mathcal{G}_\delta(t):=\int_0^t\frac{s(s+\delta)^{\frac{p-2}{2}}}{\sqrt{1+\delta+s^2}}\, ds, \qquad\mbox{ for every }t\ge0.
			$$
			Moreover the following estimate
			\begin{eqnarray}\label{DGdeltaDu}
				&&\int_{Q_\frac{R}{16}}\left|D\left[\mathcal{G}_\delta\left(\left(\left|Du\right|-\delta -1\right)_+\right)\right]\right|^2\,dz\cr\cr 
				&\le& \frac{c\left(n, p\right)}{R^2{\delta^2}}\left[\int_{Q_{ {R}}}\left(\left|Du\right|^{p}+1\right)\,dz+\frac{1}{\delta^p}\int_{Q_{R}}\left|f\right|^2 \,dz\right], 
			\end{eqnarray}
			holds for any {$R>0$ such that $Q_R=Q_R\left(z_{0}\right)\Subset\Omega_T$}.
		\end{thm}
		
		{As already mentioned, the weak solutions of \eqref{eq1*} are not twice differentiable, and hence it is not possible in general to differentiate the equation to  estimate  the second derivative of the solutions. We overcome this difficulty by introducing a suitable family of approximating problems whose solutions are  regular enough by the standard theory (\cite{DiB}). The major effort in the proof of previous Theorem is to establish suitable estimates for the solutions of the regularized problems that are uniform with respect to the approximation's parameter. Next, we take advantage from these uniform estimates in the use of a comparison argument aimed to bound the difference quotient of a suitable nonlinear function of the gradient of the solution that vanishes in the set $\Set{\left|Du\right|\le 1+\delta}$, with $\delta>0$.}

		{Roughly speaking, due to the weakness of our assumption on the datum, we only get the  higher differentiability of a nonlinear function of the gradient of the solutions that vanishes in a set which is larger with respect to that of the degeneracy of the problem. This is quite predictable, since the same kind of phenomenon occurs in the setting of widely degenerate elliptic problems (see, for example \cite{CGGP1}).
			\\
			Anyway, as a consequence of the higher differentiability result in Theorem \ref{Theorem1}, we establish a higher integrability result for the spatial gradient of the solution to equation \eqref{eq1*}, which is the following}
		
		\begin{thm}\label{thm2} Under the assumptions of Theorem \ref{Theorem1}, we have
			$$Du\in L^{p+\frac{4}{n}}_{\mathrm{loc}}\left(\Omega_T\right)$$
			with the following estimate
			\begin{eqnarray}\label{est2}
				&&\int_{Q_{\frac{\rho}{2}}}\left|Du\right|^{p\,+\,\frac{4}{n}}\,dz\leq \frac{c\left(n,p\right)}{\rho^{\frac{2(n+2)}{n}}}\left[\int_{Q_{{2\rho}}}\left(1+\left|Du\right|^{p}+\left|f\right|^2\right)\,dz\right]^{\frac{2}{n}+1},
			\end{eqnarray}
			for every parabolic cylinder $Q_{{2\rho}}\left(z_{0}\right)\Subset\Omega_T$, with a constant $c=c(n,p).$
		\end{thm}
		
		{The proof of previous Theorem consists in using an interpolation argument with the aim of  establishing  an estimate for the  $L^{p+\frac{4}{n}}$ norm of the gradient of the solutions to the approximating problems that is preserved  in the passage to the limit. }
		
		We conclude  mentioning that the elliptic version of our  equation naturally
		arises in optimal transport problems with congestion effects, and
		the regularity properties of its weak solutions have been widely investigated
		(see e.g. \cite{Ambr, Bog, Br0, Br1}).
		Moreover, we'd like to stress that, for sake of clarity, we confine ourselves to equation \eqref{eq1*}, but we believe that our techniques  apply as well  to a general class of equations with a widely degenerate
		structure.
		\medskip
		
		\section{Notations and preliminaries\label{sec:prelim}}
		
		\noindent $\hspace*{1em}$In this paper we shall denote by $C$ or
		$c$ a general positive constant that may vary on different occasions.
		Relevant dependencies on parameters will be
		properly stressed using parentheses or subscripts. The norm we use
		on $\mathbb{R}^{n}$ will be the standard Euclidean one and it will
		be denoted by $\left|\,\cdot\,\right|$. In particular, for the vectors
		$\xi,\eta\in\mathbb{R}^{n}$, we write $\langle\xi,\eta\rangle$ for
		the usual inner product and $\left|\xi\right|:=\langle\xi,\xi\rangle^{\frac{1}{2}}$
		for the corresponding Euclidean norm.\\
		$\hspace*{1em}$For points in space-time, we will use abbreviations
		like $z=(x,t)$ or $z_{0}=\left(x_{0},t_{0}\right)$, for spatial variables $x$,
		$x_{0}\in\mathbb{R}^{n}$ and times $t$, $t_{0}\in\mathbb{R}$. We
		also denote by $B\left(x_{0},\rho\right)=B_{\rho}\left(x_{0}\right)=\Set{x\in\mathbb{R}^{n}:\left|x-x_{0}\right|<\rho} $
		the open ball with radius $\rho>0$ and center $x_{0}\in\mathbb{R}^{n}$;
		when not important, or clear from the context, {we shall omit to indicate
			the center, denoting}: $B_{\rho}\equiv B\left(x_{0},\rho\right)$. Unless otherwise
		stated, different balls in the same context will have the same center.
		Moreover, we use the notation
		\[
		Q_{\rho}\left(z_{0}\right):=B_{\rho}\left(x_{0}\right)\times\left(t_{0}-\rho^{2},t_{0}\right),\qquad z_{0}=\left(x_{0},t_{0}\right)\in\mathbb{R}^{n}\times\mathbb{R},\quad\rho>0,
		\]
		for the backward parabolic cylinder with vertex $\left(x_{0},t_{0}\right)$ and
		width $\rho$. We shall sometimes omit the dependence on the vertex
		when the cylinders occurring share the same vertex.
		Finally, for a cylinder $Q=A\times\left(t_{1},t_{2}\right)$, where $A\subset\mathbb{R}^{n}$
		and $t_{1}<t_{2}$, we denote by
		\[
		\partial_{\mathrm{par}}Q:=\left(A\times\Set{t_{1}}\right)\cup\left(\partial A\times\left[t_{1},t_{2}\right]\right)
		\]
		the usual {parabolic boundary} of $Q$, which is nothing but
		the standard topological boundary without the upper cap $A\times\Set{t_{2}}$.\\
		\\
		\noindent We now recall some tools that will be useful to prove
		our results.\\ 
		For the auxiliary function $H_{\lambda}:\mathbb{R}^{n}\rightarrow\mathbb{R}^{n}$
		defined as
		\begin{equation}\label{Hdef}
			H_{\lambda}(\xi):=\begin{cases}
				\displaystyle\left(\left|\xi\right|-1\right)_{+}^{\lambda}\frac{\xi}{\left|\xi\right|}\qquad&\mbox{ if }\quad\xi\in\mathbb{R}^{n}\setminus\left\{ 0\right\},\vspace{11pt}\\
				0\qquad&\mbox{ if }\quad\xi=0,
			\end{cases}
		\end{equation}
		where $\lambda>0$ is a parameter, we record the following estimates
		(see \cite[Lemma 4.1]{Br}):\\
		\begin{lemma}\label{lem:Brasco}
			\noindent If $2\leq p<\infty$, then for every
			$\xi,\eta\in\mathbb{R}^{n}$ it holds 
			\[
			\langle H_{p-1}(\xi)-H_{p-1}(\eta),\xi-\eta\rangle\,\geq\,\frac{4}{p^{2}}\left|H_{\frac{p}{2}}(\xi)-H_{\frac{p}{2}}(\eta)\right|^{2},
			\]
			\[
			\left|H_{p-1}(\xi)-H_{p-1}(\eta)\right|\,\leq\,(p-1)\left(\left|H_{\frac{p}{2}}(\xi)\right|^{\frac{p-2}{p}}+\left|H_{\frac{p}{2}}(\eta)\right|^{\frac{p-2}{p}}\right)\left|H_{\frac{p}{2}}(\xi)-H_{\frac{p}{2}}(\eta)\right|.
			\]
		\end{lemma}
		we record the following estimates
		(see \cite[Lemma 2.8]{Bog})
		{\begin{lemma}\label{diff-A}
				Let  $\xi,\eta\in\R^{k}$ with $|\xi|>1$. Then, we have 
				\begin{eqnarray*}
					\left|H_{p-1}(\xi) - H_{p-1}(\eta)\right|
					\le
					c(p)\,
					\frac{\left[\left(\left|\xi\right|-1\right) +\left(\left|\eta\right|-1\right)_+\right]^{p-1}}{\left|\xi\right|-1}
					\left|\xi - \eta\right|
				\end{eqnarray*}
				and 
				\begin{eqnarray*}
					\langle H_{p-1}(\eta)-H_{p-1}(\xi),
					\cdot\eta-\xi\rangle\ge 
					\frac{\min\Set{1,p-1}}{2^{p+1}}
					\frac{\left(\left|\xi\right|-1\right)^p}{\left|\xi\right|\left(\left|\xi\right|+\left|\eta\right|\right)}
					\left|\eta-\xi\right|^2.
				\end{eqnarray*}
			\end{lemma}
		}
		\begin{definition}
			\noindent With the use of \eqref{Hdef}, a function $u\in C^{0}\left(0,T;L^{2}\left(\Omega\right)\right)\cap L^{p}\left(0,T;W^{1,p}\left(\Omega\right)\right)$
			is a weak solution of equation \eqref{eq1*} if 
			\begin{equation}\label{eq1H}
				\int_{\Omega_{T}}\left(u\cdot\partial_{t}\varphi-\langle H_{p-1}\left(Du\right),D\varphi\rangle\right)\,dz=-\int_{\Omega_{T}}f\varphi\,dz
			\end{equation}
			for every $\varphi\in C_{0}^{\infty}(\Omega_{T})$.\smallskip{}
		\end{definition}
		
		\noindent In the following, we shall also use the well known
		auxiliary function $V_{p}:\mathbb{R}^{n}\rightarrow\mathbb{R}^{n}$
		defined as
		{\[
			V_{p}(\xi):=\left(1+\left|\xi\right|^2\right)^{\frac{p-2}{4}}\xi,
			\]
			where $p\geq2$. We have the following
			result. 
			\begin{lemma}
				\noindent \label{lem:Lind} For every $\xi,\eta\in\mathbb{R}^{n}$
				there hold
				\begin{eqnarray*}
					\frac{1}{c_1\left(p\right)}\left|V_{p}(\xi)-V_{p}(\eta)\right|^{2}&\leq&\left(1+\left|\xi\right|^2+\left|\eta\right|^2\right)^\frac{p-2}{2}\left|\xi-\eta\right|^2\cr\cr&\le& c_1(p)\left<\left(1+\left|\xi\right|^2\right)^\frac{p-2}{2}\xi-\left(1+\left|\eta\right|^2\right)^\frac{p-2}{2}\eta, \xi-\eta\right>,
				\end{eqnarray*}
			\end{lemma}
			\noindent We refer to \cite[Chapter 12]{Lind} or to \cite[Lemma 9.2]{Giu} for a proof of these
			fundamental inequalities.} \\
		
		
		\noindent $\hspace*{1em}$For further needs, we also record the following
		interpolation inequality whose proof can be found in \cite[Proposition 3.1]{DiBenedettobook}
		
		\begin{lemma}
			
			\noindent \label{lem:inter} Assume that the function $v:Q_{r}(z_{0})\cup\partial_{\mathrm{par}}Q_{r}(z_{0})\rightarrow\mathbb{R}$
			satisfies\\
			
			$$v\in L^{\infty}\left(t_{0}-r^{2},t_{0};L^{q}\left(B_{r}\left(x_{0}\right)\right)\right)\cap L^{p}\left(t_{0}-r^{2},t_{0};W^{1,p}_0\left(B_{r}\left(x_{0}\right)\right)\right)$$\medskip{}
			
			for some exponents $1\leq p$, $q<\infty\,$.
			Then the following estimate
			\[
			\int_{Q_{r}(z_{0})}\left|v\right|^{p+\frac{pq}{n}}\,dz\leq c\left(\sup_{s\in(t_{0}-r^{2},t_{0})}\int_{B_{r}(x_{0})}\left|v(x,s)\right|^{q}\,dx\right)^{\frac{p}{n}}\int_{Q_{r}(z_{0})}\left|Dv\right|^{p}\,dz
			\]
			holds true for a positive constant $c$ depending at most on $n$,
			$p$ and $q$.
		\end{lemma}
		
		\subsection{Difference quotients\label{subsec:DiffOpe}}
		\noindent We recall here the definition and some elementary
		properties of the difference quotients (see, for example, \cite[Chapter 8]{Giu}).\\
		\begin{definition}
			\noindent For every  function $F:\mathbb{R}^{n}\rightarrow\mathbb{R}^N$
			the {finite difference operator }in the direction $x_{s}$
			is defined by
			\[
			\tau_{s,h}F(x)=F\left(x+he_{s}\right)-F(x),
			\]
			where $h\in\mathbb{R}$, $e_{s}$ is the unit vector in the direction
			$x_{s}$ and $s\in\left\{ 1,\ldots,n\right\} $. \\
			$\hspace*{1em}$The {difference quotient} of $F$ with respect
			to $x_{s}$ is defined for $h\in\mathbb{R}\setminus\left\{ 0\right\} $
			as 
			\[
			\Delta_{s,h}F(x)=\frac{\tau_{s,h}F(x)}{h}.
			\]
		\end{definition}
		
		\noindent We shall omit the index $s$ when it is not necessary, and simply write
		$
		\tau_{h}F(x)=F(x+h)-F(x)
		$
		and
		$
		\left|\Delta_{h}F(x)\right|=\displaystyle\frac{\left|\tau_{h}F(x)\right|}{|h|}
		$
		for $h\in\R^n$. 
		\begin{proposition}\label{propdiff}
			\noindent Let $F\in W^{1,p}\left(\Omega\right)$,
			with $p\geq1$, and let us set
			\[
			\Omega_{\left|h\right|}:=\Set{ x\in\Omega:\mathrm{dist}\left(x,\partial\Omega\right)>\left|h\right| } .
			\]
			Then:\\
			\\
			$\mathrm{(}a\mathrm{)}$ $\Delta_{h}F\in W^{1,p}\left(\Omega_{\left|h\right|}\right)$
			and
			\[
			D_{i}(\Delta_{h}F)=\Delta_{h}(D_{i}F),\,\,\,\,\,for\,\,every\,\,i\in\left\{ 1,\ldots,n\right\} .
			\]
			$\mathrm{(}b\mathrm{)}$ If at least one of the functions $F$ or
			$G$ has support contained in $\Omega_{\left|h\right|}$, then
			\[
			\int_{\Omega}F\,\Delta_{h}G\,dx=-\int_{\Omega}G\,\Delta_{-h}F\,dx.
			\]
			$\mathrm{(}c\mathrm{)}$ We have 
			\[
			\Delta_{h}\left(FG\right)(x)=F\left(x+he_{s}\right)\Delta_{h}G(x)\,+\,G(x)\Delta_{h}F(x).
			\]
		\end{proposition}
		
		\noindent The next result about the finite difference
		operator is a kind of integral version of Lagrange Theorem (see \cite[Lemma 8.1]{Giu}).
		\begin{lemma}\label{lem:Giusti1} 
			If $0<\rho<R$, $\left|h\right|<\displaystyle\frac{R-\rho}{2}$,
			$1<p<+\infty$, and $F\in W^{1, p}\left(B_{R},\mathbb{R}^{N}\right)$, then
			\[
			\int_{B_{\rho}}\left|\tau_{h}F(x)\right|^{p}\,dx\leq c^{p}(n)\left|h\right|^{p}\int_{B_{R}}\left|DF(x)\right|^{p}\,dx.
			\]
			Moreover
			\[
			\int_{B_{\rho}}\left|F(x+he_{s})\right|^{p}\,dx\leq \int_{B_{R}}\left|F(x)\right|^{p}\,dx.
			\]
		\end{lemma}
		
		We conclude this section with the following fundamental
		result, whose proof can be found in \cite[Lemma 8.2]{Giu}:
		\begin{lemma}\label{lem:RappIncre} 
			Let $F:\mathbb{R}^{n}\rightarrow\mathbb{R}^{N}$, $F\in L^{p}\left(B_{R},\mathbb{R}^{N}\right)$ with $1<p<+\infty$.
			Suppose that there exist $\rho\in(0,R)$ and a constant $M>0$ such
			that 
			\[
			\sum_{s=1}^{n}\int_{B_{\rho}}\left|\tau_{s,h}F(x)\right|^{p}\,dx\leq M^{p}\left|h\right|^{p}
			\]
			for every $h$, with $\left|h\right|<\displaystyle\frac{R-\rho}{2}$. Then $F\in W^{1,p}\left(B_{\rho},\mathbb{R}^{N}\right)$ and 
			\[
			\Vert DF\Vert_{L^{p}\left(B_{\rho}\right)}\leq M.
			\]
			Moreover
			\[
			\Delta_{s,h}F\rightarrow D_{s}F\qquad\mbox{ strongly in }L_{\loc}^{p}\left(B_{R}\right),\,\,as\,\,h\rightarrow0,
			\]
			for each $s\in\left\{ 1,\ldots,n\right\} $.
		\end{lemma}
		
		\subsection{Some auxiliary functions and related algebraic inequalities}\label{Gproperties}
		In this section we introduce some auxiliary functions and we list some of their properties, that will be used in what follows.\\
		
		For any $k>1$ and for $s\in [0,+\infty)$, let us consider the function 
		\begin{equation}\label{gkdef}
			g_k(s)=\frac{s^2}{k+s^2}, 
		\end{equation}
		for which we record the following
		\begin{lemma} Let $k>1$, and let $g_k$ be the   function defined by \eqref{gkdef}. 
			Then for every  $A,B\ge 0$ the following Young's type inequality
			\begin{equation}\label{propgfinale}
				{A\cdot B [s\cdot g_k'\left((s-k)_+\right)]\le 2\sqrt{2}k\left[\alpha A^2g_k\left((s-k)_+\right)+\alpha\sigma A^2+c_\alpha B^2\right]},    
			\end{equation}
			holds  for every parameters $\alpha,\sigma>0$ with a constant $c_\alpha$ independent of $\sigma$. Moreover, there exists a constant $c_k>0,$ depending on $k$, such that
			\begin{equation}\label{g_k'est}
				sg'_k\left(\left(s^2-k\right)_+\right)\le c_k,\qquad  \forall s\ge0.
			\end{equation}
		\end{lemma}
		\begin{proof} Since
			\begin{equation}\label{g'_k}
				g'_k(s)=\frac{2ks}{\left(k+s^2\right)^2},    
			\end{equation}
			both the conclusions trivially hold for $s\le \sqrt{k}.$ Now assume that $s>\sqrt{k}$
			and note that Young's inequality implies
			\begin{eqnarray}\label{propg}
				&&A\cdot B\left[s\cdot g_k'\left((s-k)_+\right)\right]\cr\cr &=&A\cdot B\cdot s\cdot g_k'\left((s-k)_+\right)\frac{\left[\sigma+(s-k)_+\right]^{\frac{1}{2}}}{\left[\sigma+(s-k)_+\right]^{\frac{1}{2}}}\cr\cr
				&\le& \alpha A^2s\cdot g_k'\left((s-k)_+\right)\left[\sigma+(s-k)_+\right]+c_\alpha \frac{B^2s\cdot g_k'\left((s-k)_+\right)}{\left[\sigma+(s-k)_+\right]}\cr\cr
				&=&\alpha A^2 s\cdot g_k'\left((s-k)_+\right)(s-k)_+
				+\alpha\sigma A^2 s\cdot g_k'\left((s-k)_+\right)\cr\cr
				&&+c_\alpha \frac{B^2s\cdot g_k'\left((s-k)_+\right)}{\left[\sigma+(s-k)_+\right]}\cr\cr
				&\le&\alpha A^2\frac{2ks(s-k)^2_+}{\left[k+(s-k)_+^2\right]^2}+\alpha\sigma A^2\frac{2ks(s-k)_+}{\left[k+(s-k)_+^2\right]^2}\cr\cr
				&&+c_\alpha B^2\frac{2ks}{\left[k+(s-k)_+^2\right]^2}\frac{(s-k)_+}{\left[\sigma+(s-k)_+\right]},
			\end{eqnarray}
			where we used the explicit expression of $g'_k(s)$ at \eqref{g'_k}. Recalling \eqref{gkdef} and since $\displaystyle\frac{t}{k+t^2}\le 1$, from \eqref{propg} we deduce
			\begin{eqnarray}\label{propg*}
				A\cdot B\left[s\cdot g_k'\left((s-k)_+\right)\right]&\le&\alpha A^2\frac{2ks}{k+(s-k)_+^2}g_k\left((s-k)_+ \right)\cr\cr 
				&&+\alpha\sigma A^2\frac{2ks}{k+(s-k)_+^2}+c_\alpha B^2\frac{2ks}{k+(s-k)_+^2}.
			\end{eqnarray}
			Setting $h(s)=\displaystyle\frac{s}{k+(s-k)_+^2}$, we can easily check that 
			$$h(k)=1,\quad \lim_{s\to+\infty}h(s)=0,\quad \max_{s\in \left[k,+\infty\right)}h(s)=h\left(\sqrt{k^2+k}\right)=\frac{1}{2}\left(1+\sqrt{1+\frac{1}{k}}\right)<\sqrt{2}$$
			and so
			$${\frac{2ks}{k+(s-k)_+^2}\le 2\sqrt{2}k \qquad\forall s>k.}$$
			Inserting this in \eqref{propg*}, we get \eqref{propgfinale}.\\
			In order to prove \eqref{g_k'est}, let us notice that, recalling \eqref{g'_k}, we have
			$$
			sg'\left(\left(s^2-k\right)_+\right)=\frac{2ks\left(s^2-k\right)_+}{\left(k+\left(s^2-k\right)_+^2\right)^2}.
			$$
			So, since the function $sg'\left(\left(s^2-k\right)_+\right)$ is continuous in the interval $
			\Set{s\ge0| s^2>k}=\left(\sqrt{k}, +\infty\right)$
			and 
			$$\lim_{s\to+\infty}\frac{2ks\left(s^2-k\right)_+}{\left(k+\left(s^2-k\right)_+^2\right)^2}=0,$$
			then there exists a constant $c_k>0$ such that
			$$sg'\left(\left(s^2-k\right)_+\right)\le c_k\qquad\mbox{ for every }s\ge0,$$
			which is the conclusion.
		\end{proof}
		\color{black}
		For any $\delta>0$, let us define
		
		\begin{equation}\label{Gdef}
			\mathcal{G}_\delta(t):=\int_0^t\frac{s(s+\delta)^{\frac{p-2}{2}}}{\sqrt{1+\delta+s^2}}\, ds, \qquad\mbox{ for }t\ge0,
		\end{equation}
		
		and observe that
		
		\begin{equation}\label{G'}
			\mathcal{G}'_\delta(t)=\frac{t(t+\delta)^{\frac{p-2}{2}}}{\sqrt{1+\delta+t^2}}.
		\end{equation}
		
		Next Lemma relates the function $\mathcal{G}_\delta\left(\left|\xi\right|\right)$ with $H_{\frac{p}{2}}(\xi)$.
		
		\begin{lemma}\label{GHLemma}
			Let $\mathcal{G}_\delta$ be the function defined by \eqref{Gdef} and $H_{\frac{p}{2}}$ be the one defined in \eqref{Hdef} with $\lambda=\displaystyle\frac{p}{2}$. Then we have
			\begin{equation}\label{GHineq}
				\left|\mathcal{G}_\delta\left(\left(\left|\xi\right|-\delta -1\right)_+\right)-\mathcal{G}_\delta\left(\left(\left|\eta\right|-\delta -1\right)_+\right)\right|^2\le c_p\left|H_{\frac{p}{2}}\left(\xi\right)-H_{\frac{p}{2}}\left(\eta\right)\right|^2
			\end{equation}
			for any $\xi, \eta \in\R^n$. 
		\end{lemma}
		\begin{proof}{If $ |\xi|<1+\delta$ and $|\eta|<1+\delta$ there is nothing to prove. So will assume that $|\xi|>1+\delta$, and without loss of generality we may suppose that $ |\eta|\le |\xi|$.} 
			Since $\mathcal{G}_\delta(t)$ is increasing, we have
			\begin{eqnarray*}
				&& \left|\mathcal{G}_\delta\left(\left|\xi\right|-1-\delta\right)-\mathcal{G}_\delta\left(\left(\left|\eta\right|-1-\delta\right)_+\right)\right|\cr\cr&=&
				\mathcal{G}_\delta\left(\left|\xi\right|-1-\delta\right)-\mathcal{G}_\delta\left(\left(\left|\eta\right|-1-\delta\right)_+\right)\cr\cr
				&=&\int_{\left(\left|\eta\right|-1-\delta\right)_+}^{\left|\xi\right|-1-\delta}\frac{s(s+\delta)^{\frac{p-2}{2}}}{\sqrt{1+\delta+s^2}}\, ds\cr\cr
				&\le&\int_{\left(\left|\eta\right|-1-\delta\right)_+}^{\left|\xi\right|-1-\delta}(s+\delta)^{\frac{p-2}{2}}\,ds\cr\cr
				&=&\frac{2}{p}\left[\left(\left|\xi\right|-1\right)^\frac{p}{2}-\left[\left(\left|\eta\right|-\delta-1\right)_++\delta\right]^\frac{p}{2}\right]. 
			\end{eqnarray*}
				
				
				Now, it can be easily checked that
				\begin{eqnarray*}
					&&\left(\left|\xi\right|-1\right)^\frac{p}{2}-\left[\left(\left|\eta\right|-\delta-1\right)_++\delta\right]^\frac{p}{2}\cr\cr
					&=&\begin{cases}
						\left(\left|\xi\right|-1\right)^\frac{p}{2}-\delta^\frac{p}{2}\qquad&\mbox{ if }\quad \left|\xi\right|>\delta+1\mbox{ and }\left|\eta\right|\le\delta+1
						\vspace{11pt}\\
						\left(\left|\xi\right|-1\right)^\frac{p}{2}-\left(\left|\eta\right|-1\right)^\frac{p}{2}\qquad&\mbox{ if }\quad \left|\xi\right|>\delta+1\mbox{ and }\left|\eta\right|>\delta+1.
					\end{cases} 
				\end{eqnarray*}
				In the first case, we have
				\begin{eqnarray*}
					\left|\left(\left|\xi\right|-1\right)^\frac{p}{2}-\delta^\frac{p}{2}\right|&=&\left(\left|\xi\right|-1\right)^\frac{p}{2}-\delta^\frac{p}{2}\le \left(\left|\xi\right|-1\right)^\frac{p}{2}-\left(\left|\eta\right|-1\right)_+^\frac{p}{2}\cr\cr&=&\left|H_{\frac{p}{2}}\left(\xi\right)\right|-\left|H_{\frac{p}{2}}\left(\eta\right)\right|\le \left|H_{\frac{p}{2}}\left(\eta\right)-H_{\frac{p}{2}}\left(\xi\right)\right|,
				\end{eqnarray*}
				while, in the second,
				$$ \left(\left|\xi\right|-1\right)^\frac{p}{2}-\left[\left(\left|\eta\right|-\delta-1\right)_++\delta\right]^\frac{p}{2}=\left|H_{\frac{p}{2}}\left(\xi\right)\right|-\left|H_{\frac{p}{2}}\left(\eta\right)\right|\le \left|H_{\frac{p}{2}}\left(\eta\right)-H_{\frac{p}{2}}\left(\xi\right)\right|.$$
				Therefore,
				
				\begin{eqnarray*}
					\left|\mathcal{G}_\delta\left(\left(\left|\xi\right|-\delta -1\right)_+\right)-\mathcal{G}_\delta\left(\left(\left|\eta\right|-\delta -1\right)_+\right)\right|^2\le c_p\left|H_{\frac{p}{2}}\left(\xi\right)-H_{\frac{p}{2}}\left(\eta\right)\right|^2
				\end{eqnarray*}
				for every $\xi, \eta\in\R^n$, which is \eqref{GHineq}.
			\end{proof}
			Arguing as in \cite[Lemma 2.1]{EleMarMas3}, we prove the following.
			\begin{lemma}\label{estGdelta} 
				Let $0<\delta\le 1$ and $p\ge 2$. Then the following inequalities hold
				\begin{equation*}
					c_{p,\delta}(t+\delta)^{\frac{p}{2}}-\tilde{c}_{p,\delta}\le\mathcal{G}_\delta(t)\le \frac{2}{p} (t+\delta)^{\frac{p}{2}}
				\end{equation*}
				with constants $\tilde{c}_{p,\delta}$ and $c_{p,\delta}<\displaystyle\frac{2}{p}$ depending on $p$ and $\delta$.
			\end{lemma}
			\begin{proof} If $p=2$, one can easily calculate 
				$$\mathcal{G}_\delta(t)=\int_0^t\frac{s}{\sqrt{1+\delta+s^2}}\, ds=\left[\sqrt{1+\delta+s^2}\right]_{0}^{t}=\sqrt{1+\delta+t^2}-\sqrt{1+\delta},$$
				from which immediately follows
				$$ \frac12\left(t+\delta\right)-\frac12\left(\sqrt{1+\delta}+\delta\right)\le\mathcal{G}_\delta(t)\le t+\delta.$$
				Let $p>2$. The right  inequality is a simple consequence of the trivial bound $\frac{s}{\sqrt{1+\delta+s^2}}< 1$. For the left inequality we start observing that
				$$
				\sqrt{1+\delta+s^2}\le \sqrt{1+\delta}+s\quad\Longrightarrow\quad \mathcal{G}_\delta(t)\ge \int_0^t\frac{s\left(s+\delta\right)^{\frac{p-2}{2}}}{\sqrt{1+\delta}+s}\,ds.
				$$
				Now, we calculate the integral in previous formula. By the change of variable $r=\sqrt{1+\delta}+s$, we get
				\begin{eqnarray*}
					&&\int_0^t\frac{s\left(s+\delta\right)^{\frac{p-2}{2}}}{\sqrt{1+\delta}+s}\, ds=\int_{\sqrt{1+\delta}}^{t+\sqrt{1+\delta}}\frac{\left(r-\sqrt{1+\delta}\right)\left(r-\sqrt{1+\delta}+\delta\right)^{\frac{p-2}{2}}}{r}\,ds\cr\cr
					&=&\int_{\sqrt{1+\delta}}^{t+\sqrt{1+\delta}}\left(r-\sqrt{1+\delta}+\delta\right)^{\frac{p-2}{2}}\,ds-\sqrt{1+\delta}\int_{\sqrt{1+\delta}}^{t+\sqrt{1+\delta}}\frac{\left(r-\sqrt{1+\delta}+\delta\right)^{\frac{p-2}{2}}}{r}\, ds\cr\cr
					&\ge&\frac{2}{p}\left[\left(r-\sqrt{1+\delta}+\delta\right)^{\frac{p}{2}}\right]_{\sqrt{1+\delta}}^{t+\sqrt{1+\delta}}-\sqrt{1+\delta}\int_{\sqrt{1+\delta}}^{t+\sqrt{1+\delta}}\left(r-\sqrt{1+\delta}+\delta\right)^{\frac{p}{2}-2}\,ds,
				\end{eqnarray*}
				since $0<\delta\le 1$, we have $\delta\le \sqrt{1+\delta}$ and therefore $r-\sqrt{1+\delta}+\delta\le r$. Calculating the last integral in previous formula, we get
				\begin{eqnarray*}
					&&\int_0^t\frac{s(s+\delta)^{\frac{p-2}{2}}}{\sqrt{1+\delta}+s}\,ds\cr\cr 
					&\ge&\frac{2}{p}\left[\left(r-\sqrt{1+\delta}+\delta\right)^{\frac{p}{2}}\right]_{\sqrt{1+\delta}}^{t+\sqrt{1+\delta}}-\frac{2\sqrt{1+\delta}}{p-2}\left[\left(r-\sqrt{1+\delta}+\delta\right)^{\frac{p}{2}-1}\right]_{\sqrt{1+\delta}}^{t+\sqrt{1+\delta}}\cr\cr
					&=&\frac{2}{p}\left[(t+\delta)^{\frac{p}{2}}-\delta^{\frac{p}{2}}\right]-\frac{2\sqrt{1+\delta}}{p-2}\left[(t+\delta)^{\frac{p}{2}-1}-\delta^{\frac{p}{2}-1}\right]\cr\cr
					&=&\frac{2}{p}(t+\delta)^{\frac{p}{2}}-\frac{2\sqrt{1+\delta}}{p-2}(t+\delta)^{\frac{p}{2}-1}+\frac{2\sqrt{1+\delta}}{p-2}\delta^{\frac{p}{2}-1}-\frac{2}{p}\delta^{\frac{p}{2}}.
				\end{eqnarray*}
				Therefore the lemma will be proven if there exists a constant $c_{p,\delta}<\displaystyle\frac{2}{p}$ such that
				$$c_{p,\delta}(t+\delta)^{\frac{p}{2}}\le \frac{2}{p}(t+\delta)^{\frac{p}{2}}-\frac{2\sqrt{1+\delta}}{p-2}(t+\delta)^{\frac{p}{2}-1}+\frac{2\sqrt{1+\delta}}{p-2}\delta^{\frac{p}{2}-1}-\frac{2}{p}\delta^{\frac{p}{2}}$$
				which, setting 
				$$
				h(t)=\frac{2\sqrt{1+\delta}}{p-2}(t+\delta)^{\frac{p}{2}-1}+\left(c_{p,\delta}-\frac{2}{p}\right)(t+\delta)^{\frac{p}{2}},$$
				is equivalent to prove that there exists $c_{p,\delta}$ such that 
				$$h(t)\le \frac{2\sqrt{1+\delta}}{p-2}\delta^{\frac{p}{2}-1}-\frac{2}{p}\delta^{\frac{p}{2}}.
				$$
				It is easy to check that $h(t)$ attains his maximum for $t+\delta=\displaystyle\frac{2\sqrt{1+\delta}}{2-pc_{p,\delta}}$
				and so
				{$$h(t)\le h\left(\frac{2\sqrt{1+\delta}}{2-pc_{p,\delta}}-\delta\right)=\left(2\sqrt{1+\delta}\right)^{\frac{p}{2}}\left(\frac{1}{2-pc_{p,\delta}}\right)^{\frac{p-2}{2}}\frac{2}{p\left(p-2\right)}$$}
				Therefore, to complete the proof it's enough to solve the following equation
				{$$\left(2\sqrt{1+\delta}\right)^{\frac{p}{2}}\left(\frac{1}{2-pc_{p,\delta}}\right)^{\frac{p-2}{2}}\frac{2}{p\left(p-2\right)}= \frac{2\sqrt{1+\delta}}{p-2}\delta^{\frac{p}{2}-1}-\frac{2}{p}\delta^{\frac{p}{2}}$$}
				which is equivalent to
				{$$\frac{1}{2-pc_{p,\delta}}=\left(\frac{\delta}{2\sqrt{1+\delta}}\right)^\frac{p}{p-2}\left(\frac{p\left(\sqrt{1+\delta}-\delta\right)}{\delta}+2\right)^\frac{2}{p-2}$$}
				that, for $0<\delta<1$, admits a unique solution {$c_{p,\delta}<\displaystyle\frac{2}{p}$}.
			\end{proof}

			\section{The regularization}
			For $\varepsilon>0$, we introduce the  sequence of operators 
			$$A_\varepsilon(\xi):=\left(\left|\xi\right|-1\right)_+^{p-1}\frac{\xi}{\left|\xi\right|}+\varepsilon\left(1+|\xi|^2\right)^{\frac{p-2}{2}}\xi $$
			and by 
			$$u_\varepsilon\in C^0\left(t_0-R^2,t_0;L^2\left(B_R\right)\right)\cap L^p\left(t_0-R^2,t_0; u+W_0^{1,p}\left(B_{R}\right)\right)$$  
			we denote the unique solution to the corresponding problems
			\begin{equation}\label{app}
				\begin{cases}	u^\varepsilon_t-\div\left(A_\varepsilon\left(Du^\varepsilon\right)\right)=f^\varepsilon\qquad&\mbox{ in }Q_{{R}}\left(z_0\right)\vspace{11pt}\\
					u^\varepsilon=u\qquad&\mbox{ in }\partial_{\mathrm{par}} Q_{{R}}\left(z_0\right)
				\end{cases}
			\end{equation}
			where $Q_{{R}}\left(z_0\right)\Subset\Omega_T$ with $R<1$, $f^\varepsilon=f*\rho_\varepsilon$ with $\rho_\varepsilon$ the usual sequence of mollifiers. 
			One can easily   check that the operator $A_\varepsilon$  satisfies $p$-growth and $p$-ellipticity assumptions with constants depending on $\varepsilon$.\\ 
			Therefore, by the results in \cite{DMS}, we have $$V_p\left(Du^\varepsilon\right)\in L^2_{\loc}\left(0,T; W^{1,2}_\loc\left(B_{{R}}\left(x_0\right), \R^n\right)\right)\qquad\mbox{ and }\qquad \left|Du^\varepsilon\right|\in L^{p+\frac{4}{n}}_{\mathrm{loc}}\left(Q_{{R}}\right)$$
			and, by the definition of $V_p(\xi)$, 
			this yields
			\begin{equation}\label{derV}
				DV_p\left(Du^\eps\right)\approx\left(1+\left|Du^\varepsilon\right|^2\right)^{\frac{p-2}{4}}D^2u^\varepsilon\in L^{2}_{\mathrm{loc}}\left(Q_{{R}};\R^{n\times n}\right)\quad\implies\quad\left|D^2u^\varepsilon\right|\in L^{2}_{\mathrm{loc}}\left(Q_{{R}}\right)
			\end{equation}
			
			By virtue of \cite[Theorem 1.1]{AmbrosioPdN}, we also have 
			$
			H_{\frac{p}{2}}\left(Du^\eps\right)\in L^2_{\loc}\left(0,T; W^{1,2}_\loc\left(\Omega,\R^n\right)\right)
			$
			and, by the definition of $H_{\frac{p}{2}}(\xi)$, it follows
			\begin{equation}\label{derH}
				\left|DH_{\frac{p}{2}}\left(Du\right)\right|\le c_p\left(\left|Du^\varepsilon\right|-1\right)_+^\frac{p-2}{2}|D^2u^\varepsilon|\in L^{2}_{\mathrm{loc}}\left(Q_{{R}};\R^{n\times n}\right).
			\end{equation}
			\subsection{Uniform a priori estimates}
			The first step in the proof of Theorem \ref{Theorem1} is the following estimate for solutions to the regularized problem \eqref{app}.
			
			\begin{lemma}\label{uniformestlemma}
				Let $u^\varepsilon\in C^0\left(t_0-{R^2},t_0;L^2\left(B_{{R}}\right)\right)\cap L^p\left(t_0-{R^2},t_0; u+W_0^{1,p}\left(B_{{R}}\right)\right)$ be the unique solution to \eqref{app}. Then the following estimate
				\begin{eqnarray}\label{uniformest}
					&&\sup_{\tau\in\left(t_0-{{4\rho^2}}, t_0\right)}\int_{B_{\rho}}\left(\left|Du^\varepsilon(x,\tau)\right|^2-1-\delta\right)_+\,dx\cr\cr
					&&+\int_{Q_\rho}\left|D\left[\mathcal{G}_\delta\left(\left(\left|Du^\varepsilon\right|-\delta -1\right)_+\right)\right]\right|^2\,dz\cr\cr
					&\le&\frac{c}{\rho^2}\left[\int_{Q_{2\rho}}\left(1+\left|D u^\varepsilon\right|^p\right)\, dz +\delta^{2-p}\int_{Q_{2\rho}}\left|f^\varepsilon\right|^2 \,dz\right]
				\end{eqnarray}
				holds for any $\varepsilon\in(0, 1]$ and for every ${Q_\rho\Subset Q_{2\rho}\Subset Q_R}$, with a constant $c=c(n,p)$ independent of $\varepsilon$.
			\end{lemma}
			\begin{proof}
				The  weak formulation of  \eqref{app} reads as
				\begin{equation*}
					\int_{Q_{{R}}}\left(u^\varepsilon\cdot\partial_t\varphi-\langle A_{\varepsilon}\left(Du^\varepsilon\right), D\varphi\rangle\right)\,dz=-\int_{Q_{{R}}}f^\varepsilon\cdot\varphi\,dz
				\end{equation*}
				for any test function $\varphi\in C^{\infty}_0\left(Q_{{R}}\right)$.
				Recalling the notation used in \eqref{eq1H}, and replacing $\varphi$ with $\Delta_{-h}\varphi=\displaystyle\frac{\tau_{-h}\varphi}{h}$ for a sufficiently small $h\in\R\setminus\set{0}$, by virtue of the properties of difference quotients, we have
				
				\begin{eqnarray}\label{weakeq1*}
					&&	\int_{Q_{{R}}}\left(\Delta_{h}u^\varepsilon\cdot\partial_t\varphi-\left<\Delta_{h}H_{p-1}\left(Du^\varepsilon\right), D\varphi\right>-\varepsilon\left<\Delta_{h}\left(\left(1+\left|Du^\varepsilon\right|^2\right)^{\frac{p-2}{2}}Du^\varepsilon\right), D\varphi\right>\right)\,dz\cr\cr
					&=&-\int_{Q_{{R}}}f^\varepsilon\cdot\Delta_{-h}\varphi\,dz.
				\end{eqnarray}
				Arguing as  in \cite[Lemma 5.1]{DMS}, from \eqref{weakeq1*} we get
				\begin{eqnarray*}\label{weakeq1bis*}
					&&	\int_{Q_{{R}}}\partial_t\Delta_{h}u^\varepsilon\cdot\varphi\,dz+\int_{Q_{{R}}}\left<\Delta_{h}H_{p-1}\left(Du^\varepsilon\right), D\varphi\right>\,dz\cr\cr
					&&+\varepsilon\int_{Q_{{R}}}\left<\Delta_{h}\left(\left(1+\left|Du^\varepsilon\right|^2\right)^{\frac{p-2}{2}}Du^\varepsilon\right), D\varphi\right>\,dz=\int_{Q_{{R}}}f^\varepsilon\cdot\Delta_{-h }\varphi\,dz.
				\end{eqnarray*}
				For $\Phi\in W^{1,\infty}_0\left(Q_{{R}}\right)$ non negative and $g\in W^{1,\infty}\left(\mathbb{R}\right)$ non negative and non decreasing, we choose $\varphi=\Phi\cdot\Delta_h u^\eps\cdot g\left(\left|\Delta_h u^\varepsilon\right|^2\right)$ in previous identity, thus getting
				\begin{eqnarray*}\label{weakeq2*}
					&&\int_{Q_{{R}}}\partial_t\left(\Delta_{h}u^\varepsilon\right)\Delta_h u^\varepsilon \cdot g\left(\left|\Delta_h u^\varepsilon\right|^2\right)\Phi\,dz\cr\cr &&+\int_{Q_{{R}}}\left<\Delta_hH_{p-1}\left(Du^\varepsilon\right), D\left[\Phi\Delta_h u^\varepsilon g\left(\left|\Delta_h u^\varepsilon\right|^2\right)\right]\right>\,dz\cr\cr
					&&+\varepsilon\int_{Q_{{R}}}\left<\Delta_{h}\left(\left(1+\left|Du^\varepsilon\right|^2\right)^{\frac{p-2}{2}}Du^\varepsilon\right), D\left[\Phi\Delta_h u g\left(\left|\Delta_h u^\varepsilon\right|^2\right)\right]\right>\,dz\cr\cr
					&=&\int_{Q_{{R}}}f^\varepsilon\cdot\Delta_{-h}\left(\Phi\Delta_h u^\varepsilon\cdot g\left(\left|\Delta_h u^\varepsilon\right|^2\right)\right)\,dz,
				\end{eqnarray*}
				i.e.
				\begin{eqnarray}\label{weakeq2bis*}
					&&\int_{Q_{{R}}}\partial_t\left(\Delta_{h}u^\varepsilon\right)\Delta_h u^\varepsilon \cdot g\left(\left|\Delta_h u^\varepsilon\right|^2\right)\Phi\,dz\cr\cr
					&&+\int_{Q_{{R}}} \Phi\left<\Delta_hH_{p-1}\left(Du^\varepsilon\right),\Delta_h Du^\varepsilon\cdot g\left(\left|\Delta_h u^\varepsilon\right|^2\right)\right>\,dz\cr\cr
					&&+\varepsilon\int_{Q_{{R}}} \Phi\left<\Delta_{h}\left(\left(1+\left|Du^\varepsilon\right|^2\right)^{\frac{p-2}{2}}Du^\varepsilon\right),\Delta_h Du^\varepsilon\cdot g\left(\left|\Delta_h u^\varepsilon\right|^2\right)\right>\,dz\cr\cr
					&&+2\int_{Q_{{R}}}\Phi\left<\Delta_hH_{p-1}\left(Du^\varepsilon\right), \left|\Delta_h u^\varepsilon\right|^2\Delta_h Du^\varepsilon \cdot g'\left(\left|\Delta_h u^\varepsilon\right|^2\right)\right>\,dz\cr\cr
					&&+2\varepsilon\int_{Q_{{R}}}\Phi\left<\Delta_{h}\left(\left(1+|Du^\varepsilon|^2\right)^{\frac{p-2}{2}}Du^\varepsilon\right), \left|\Delta_h u^\varepsilon\right|^2\Delta_h Du^\varepsilon \cdot g'\left(\left|\Delta_h u^\varepsilon\right|^2\right)\right>\,dz\cr\cr
					&=&-\int_{Q_{{R}}}\left<\Delta_hH_{p-1}\left(Du^\varepsilon\right), D\Phi\cdot\Delta_h u^\varepsilon\cdot g\left(\left|\Delta_h u^\varepsilon\right|^2\right)\right>\,dz\cr\cr
					&&-\varepsilon\int_{Q_{{R}}}\left<\Delta_{h}\left(\left(1+\left|Du^\varepsilon\right|^2\right)^{\frac{p-2}{2}}Du^\varepsilon\right), D\Phi\cdot\Delta_h u^\varepsilon\cdot g\left(\left|\Delta_h u^\varepsilon\right|^2\right)\right>\,dz\cr\cr
					&&+\int_{Q_{{R}}}f^\varepsilon\cdot\Delta_{-h}\left(\Phi\Delta_h u^\varepsilon\cdot g\left(\left|\Delta_h u^\varepsilon\right|^2\right)\right)\,dz,
				\end{eqnarray}
				that we rewrite  as follows 
				$$J_{h, 1}+J_{h, 2}+J_{h, 3}+J_{h, 4}+J_{h, 5}=-J_{h, 6}-J_{h, 7}+J_{h, 8}.$$
				Arguing as in \cite{BDM},the first integral in equation \eqref{weakeq2bis*}  can be expressed as follows
				\begin{eqnarray*}
					J_{h, 1}&=&\int_{Q_{{R}}}\partial_t\left(\Delta_{h}u^\varepsilon\right)\Delta_h u^\varepsilon \cdot g\left(\left|\Delta_h u^\varepsilon\right|^2\right)\Phi\,dz=\frac{1}{2}\int_{Q_{{R}}}\partial_t\left(\left|\Delta_{h}u^\varepsilon\right|^2\right)\cdot g\left(\left|\Delta_h u^\varepsilon\right|^2\right)\Phi\,dz\cr\cr
					&=&\frac{1}{2}\int_{Q_{{R}}}\partial_t\left(\int_0^{\left|\Delta_{h}u^\varepsilon\right|^2}g(s)\,ds\right)\Phi\,dz=-\frac{1}{2}\int_{Q_{{R}}}\left(\int_0^{\left|\Delta_{h}u^\varepsilon\right|^2}g(s)\,ds\right)\partial_t\Phi\,dz.
				\end{eqnarray*}
				
				Using Lemma \ref{diff-A}, since $\Phi,g$ are non negative, we have
				$$J_{h, 2}\ge \int_{Q_{{R}}}\Phi\cdot g\left(\left|\Delta_h u^\varepsilon\right|^2\right){\left|\Delta_h Du^\varepsilon\right|^2\frac{\left(\left|Du^\varepsilon\right|-1\right)^p}{\left|Du^\varepsilon\right|\left(\left|Du^\varepsilon\right|+\left|Du^\varepsilon(x+h)\right|\right)}}\,dz.$$
				The right inequality in the assertion of Lemma \ref{lem:Lind} yields
				$$J_{h, 3}\ge \varepsilon c_p\int_{Q_{{R}}}\Phi\cdot g\left(\left|\Delta_h u^\varepsilon\right|^2\right)\left|\Delta_h V_{p}\left(Du^\varepsilon\right)\right|^2\,dz$$
				Moreover, again by Lemmas \ref{diff-A} and \ref{lem:Lind} and the fact that $g'(s)\ge 0$, we infer 
				$$J_{h, 4}+J_{h, 5}\ge0.$$
				Therefore \eqref{weakeq2bis*} implies
				\begin{eqnarray}\label{eq2*}
					&&-\frac{1}{2}\int_{Q_{{R}}}\left(\int_0^{\left|\Delta_{h}u^\varepsilon\right|^2}g(s)\,ds\right)\partial_t\Phi\,dz\cr\cr
					&&+\int_{Q_{{R}}}\Phi\cdot g\left(\left|\Delta_h u^\varepsilon\right|^2\right){\left|\Delta_h Du^\varepsilon\right|^2\frac{\left(\left|Du^\varepsilon\right|-1\right)^p}{\left|Du^\varepsilon\right|\left(\left|Du^\varepsilon\right| +\left|Du^\varepsilon(x+h)\right|\right)}}\,dz\cr\cr
					&&+c_p\varepsilon\int_{Q_{{R}}}\Phi\cdot g\left(\left|\Delta_h u^\varepsilon\right|^2\right)\left|\Delta_hV_{p}\left(Du^\varepsilon\right)\right|^2\,dz\cr\cr
					&\le&\int_{Q_{{R}}}\left|D\Phi\right|\left|\Delta_{h}H_{p-1}\left(Du^\varepsilon\right)\right|\left|\Delta_h u^\varepsilon\right| \cdot g\left(\left|\Delta_h u^\varepsilon\right|^2\right)\,dz\cr\cr
					&&+\varepsilon\int_{Q_{{R}}}\left|D\Phi\right|\left|\Delta_{h}\left(\left(1+\left|Du^\varepsilon\right|^2\right)^{\frac{p-2}{2}}Du^\varepsilon\right)\right|\left|\Delta_h u^\varepsilon\right| \cdot g\left(\left|\Delta_h u^\varepsilon\right|^2\right)\,dz\cr\cr
					&&+\int_{Q_{{R}}}\left|f^\varepsilon\right|\left|\Delta_{-h}\left(\Phi\Delta_h u^\varepsilon\cdot g\left(\left|\Delta_h u^\varepsilon\right|^2\right)\right)\right|\,dz.
				\end{eqnarray}
				Now let us consider a parabolic cylinder $Q_{{\rho}}\left(z_0\right)\Subset Q_{{2\rho}}\left(z_0\right)\Subset Q_{{R}}\left(z_0\right)$ with $\rho<{2\rho}<{R}$ and $t_0>0$. For a fixed time $\tau\in\left(t_0-{4\rho^2}, t_0\right)$ and $\theta\in\left(0, t_0-\tau\right)$, we choose $\Phi(x,t)=\eta^2(x)\chi(t)\tilde\chi(t)$ with $\eta\in C_0^\infty\left(B_{{2\rho}}\left(x_0\right)\right)$, $0\le\eta\le 1$, $\chi\in W^{1,\infty}\left(\left[0,T\right]\right)$ with $\partial_t\chi\ge0$ and $\tilde\chi$ a Lipschitz continuous function defined, for $0<\tau<\tau+\theta<T$, as follows
				\begin{equation*}
					\tilde\chi(t)=\begin{cases}
						1\qquad&\mbox{ if }\quad t\le \tau\vspace{11pt}\\
						1-\displaystyle\frac{t-\tau}{\theta}\qquad&\mbox{ if }\quad \tau<t\le \tau+\theta\vspace{11pt}\\
						0\qquad&\mbox{ if }\quad\tau+\theta< t\le T
					\end{cases}
				\end{equation*}
				so that \eqref{eq2*} yields 
				{\begin{eqnarray}\label{eq22*}
						I_{h, 1}+I_{h, 2}+I_{h, 3}&:=&\frac{1}{2}\int_{B_{{2\rho}}}\eta^2\chi(\tau)\left(\int_0^{\left|\Delta_{h}u^\varepsilon(x,\tau)\right|^2}g(s)\,ds\right)\,dx\cr\cr
						&&+{c_p}\int_{Q^{\tau}}\eta^2\chi(t)\cdot g\left(\left|\Delta_h u^\varepsilon\right|^2\right){\left|\Delta_h Du^\varepsilon\right|^2\frac{\left(\left|Du^\varepsilon\right|-1\right)^p}{\left|Du^\varepsilon\right|\left(\left|Du^\varepsilon\right| +\left|Du^\varepsilon(x+h)\right|\right)}}\,dz\cr\cr
						&&+c_p\varepsilon\int_{{Q^{\tau}}}\eta^2\chi(t)g\left(\left|\Delta_h u^\varepsilon\right|^2\right)\left|\Delta_hV_{p}\left(Du^\varepsilon\right)\right|^2\,dz\cr\cr
						&\le&2\int_{Q^{\tau}}\eta\chi(t)\left|D\eta\right|\left|\Delta_{h}H_{p-1}\left(Du^\varepsilon\right)\right|\left|\Delta_h u^\varepsilon\right| \cdot g\left(\left|\Delta_h u^\varepsilon\right|^2\right)\,dz\cr\cr
						&&+2\varepsilon	\int_{Q^{\tau}}\eta\chi(t)\left|D\eta\right|\left|\Delta_{h}\left(\left(1+\left|Du^\varepsilon\right|^2\right)^{\frac{p-2}{2}}Du^\varepsilon\right)\right|\left|\Delta_h u^\varepsilon\right| \cdot g\left(\left|\Delta_h u^\varepsilon\right|^2\right)\,dz\cr\cr
						&&+\int_{Q^{\tau}}\chi(t)\left|f^\varepsilon\right|\left|\Delta_{-h}\left(\eta^2\Delta_h u^\varepsilon\cdot g\left(\left|\Delta_h u^\varepsilon\right|^2\right)\right)\right|\,dz\cr\cr
						&&+\frac{1}{2}\int_{Q^{\tau}}\eta^2\partial_t\chi(t)\left(\int_0^{\left|\Delta_{h}u^\varepsilon\right|^2}g(s)\,ds\right)\,dz\cr\cr
						&=:&I_{h, 4}+I_{h, 5}+I_{h, 6}+I_{h, 7},
					\end{eqnarray}
					where we used the notation $Q^\tau=B_{{2\rho}}\left(x_0\right)\times \left(t_0-{4\rho^2},\tau\right).$\\
					Since $g\in W^{1,\infty}\left(\left[0, \infty\right)\right)$, by  \eqref{derV}, by the last assertion of Lemma \ref{lem:RappIncre} {and by Fatou's Lemma}, we have
					\begin{eqnarray}\label{I_{h, 1}I_{h, 2}I_{h, 3}}
						&&{\liminf_{h\to0}}\left(I_{h, 1}+I_{h, 2}+I_{h, 3}\right)\cr
						&\le &\frac{1}{2}\int_{B_{{2\rho}}}\eta^2\chi(\tau)\left(\int_0^{\left|Du^\varepsilon(x,\tau)\right|^2}g(s)\,ds\right)\,dx\cr\cr
						&&+{c_p}\int_{Q^{\tau}}\eta^2\chi(t)\cdot g\left(\left|D u^\varepsilon\right|^2\right){\left|D^2u^\varepsilon\right|^2\frac{\left(\left|Du^\varepsilon\right|-1\right)^p}{\left|Du^\varepsilon\right|^2}}\,dz\cr
						&&+c_p\varepsilon\int_{{Q^{\tau}}}\eta^2\chi(t)g\left(\left|D u^\varepsilon\right|^2\right)\left|DV_{p}\left(Du^\varepsilon\right)\right|^2\,dz.
					\end{eqnarray}
					and 
					\begin{equation}\label{I_{h, 7}}
						\lim_{h\to0}I_{h, 7}=\frac{1}{2}\int_{Q^{\tau}}\eta^2\partial_t\chi(t)\left(\int_0^{\left|Du^\varepsilon\right|^2}g(s)\,ds\right)\,dz.
					\end{equation}
					Now let us observe that
					\begin{equation}\label{DH_p-1}
						\left|DH_{p-1}\left(Du^\varepsilon\right)\right|\le c_p\left(\left|Du^\varepsilon\right|-1\right)^{p-2}_+\left|D^2u^\varepsilon\right|   
					\end{equation}
					and, 
					using H\"{o}lder's inequality with exponents $\left(\frac{2\left(p-1\right)}{p-2}, \frac{2\left(p-1\right)}{p}\right)$,  we have
					\begin{eqnarray*}
						\int_{B_R}\left|DH_{p-1}\left(Du^\varepsilon\right)\right|^\frac{p}{p-1}\,dx &\le& c_p\int_{B_R}\left[\left(\left|Du^\varepsilon\right|-1\right)^{p-2}_+\left|D^2u^\varepsilon\right|\right]^\frac{p}{p-1}dx\cr\cr
						&\le&c_p\left(\int_{B_R}\left(\left|Du^\varepsilon\right|-1\right)^{p}_+dx\right)^\frac{p-2}{2\left(p-1\right)}\cr\cr&&\quad\cdot\left(\int_{B_R}\left[\left(\left|Du^\varepsilon\right|-1\right)^{\frac{p-2}{2}}_+\left|D^2u^\varepsilon\right|\right]^2dx\right)^\frac{p}{2\left(p-1\right)}, 
					\end{eqnarray*}
					and since, by \eqref{derH}, the right hand side of previous inequality is finite again by Lemma \ref{lem:RappIncre}, we have
					$$
					\Delta_hH_{p-1}\left(Du^\varepsilon\right)\to DH_{p-1}\left(Du^\varepsilon\right)\qquad\mbox{ strongly in }\qquad L^2\left(0, T; L^\frac{p}{p-1}\left(B_R\right)\right)\qquad\mbox{ as }h\to0,
					$$
					which, since
					$\Delta_{h}u^\varepsilon\to Du^\varepsilon $ strongly in $L^2\left(0, T; L^p\left(B_R\right)\right)$ as $h\to 0,
					$
					implies
					\begin{equation}\label{I_{h, 4}}
						\lim_{h\to0}I_{h, 4}=2\int_{Q^{\tau}}\eta\chi(t)\left|D\eta\right|\left|DH_{p-1}\left(Du^\varepsilon\right)\right|\left|Du^\varepsilon\right|g\left(\left|Du^\varepsilon\right|^2\right)\,dz.
					\end{equation}
					Using similar arguments,
					we can check that
					\begin{equation}\label{I_{h, 5}}
						\lim_{h\to0}I_{h, 5}=2\varepsilon	\int_{Q^{\tau}}\eta\chi(t)\left|D\eta\right|\left|D\left(\left(1+\left|Du^\varepsilon\right|^2\right)^{\frac{p-2}{2}}Du^\varepsilon\right)\right|\left|D u^\varepsilon\right| \cdot g\left(\left|D u^\varepsilon\right|^2\right)\,dz.
					\end{equation}
					Now, by Proposition \ref{propdiff}(c), it holds
					\begin{eqnarray*}
						\left|\Delta_{-h}\left(\eta^2\Delta_h u^\varepsilon\cdot g\left(\left|\Delta_h u^\varepsilon\right|^2\right)\right)\right|
						&\le&c\Vert D\eta\Vert_\infty\left|\Delta_h u^\varepsilon\right|\left|g\left(\left|\Delta_h u^\varepsilon\right|^2\right)\right|\cr\cr
						&&+c\left|\Delta_{-h}\left(\Delta_h u^\varepsilon\right)\right|\left|g\left(\left|\Delta_h u^\varepsilon\right|^2\right)\right|\cr\cr
						&&+c\left|\Delta_h u^\varepsilon\right|^2\left|g'\left(\left|\Delta_h u^\varepsilon\right|^2\right)\right|\left|\Delta_hDu^\varepsilon\right|.
					\end{eqnarray*}
					and choosing $g$ such that  
					\begin{equation}\label{**}
						sg'\left(s^2\right)\le M, 
					\end{equation}
					for a positive constant $M$,
					we have
					\begin{eqnarray}\label{critical_g}
						\left|\Delta_{-h}\left(\eta^2\Delta_h u^\varepsilon\cdot g\left(\left|\Delta_h u^\varepsilon\right|^2\right)\right)\right|&\le&
						c\Vert D\eta\Vert_\infty\left|\Delta_h u^\varepsilon\right|\left|g\left(\left|\Delta_h u^\varepsilon\right|^2\right)\right|\cr\cr
						&&+c\left|\Delta_{-h}\left(\Delta_h u^\varepsilon\right)\right|\left|g\left(\left|\Delta_h u^\varepsilon\right|^2\right)\right|\cr\cr
						&&+c_M\left|\Delta_{h} u^\varepsilon\right|\left|\Delta_{-h}Du^\varepsilon\right|
					\end{eqnarray}
					Since
					$\Delta_{h}u^\varepsilon\to Du^\varepsilon$, $\Delta_{-h}\left(\Delta_h u^\varepsilon\right)\to D^2u^\varepsilon$, $\Delta_{-h}Du^\varepsilon\to D^2u^\varepsilon$ strongly in $L^{2}\left(0,T; L^2_\loc\left(\Omega\right)\right)$ as $h\to0$, and $f^\varepsilon\in C^\infty\left(\Omega_T\right)$, 
					thanks to \eqref{critical_g}, we have
					\begin{equation}\label{I_{h, 6}}
						\lim_{h\to0}I_{h, 6}=\int_{Q^{\tau}}\chi(t)\left|f^\varepsilon\right|\left|D\left(\eta^2Du^\varepsilon\cdot g\left(\left|Du^\varepsilon\right|^2\right)\right)\right|\,dz.
					\end{equation}
					So, collecting \eqref{I_{h, 1}I_{h, 2}I_{h, 3}}, \eqref{I_{h, 7}}, \eqref{I_{h, 4}}, \eqref{I_{h, 5}} and \eqref{I_{h, 6}}, we can pass to the limit as $h\to0$ in \eqref{eq22*}, thus getting
					\begin{eqnarray}\label{eq22lim}
						&&\frac{1}{2}\int_{B_{{2\rho}}}\eta^2\chi(\tau)\left(\int_0^{\left|Du^\varepsilon(x,\tau)\right|^2}g(s)\,ds\right)\,dx\cr\cr
						&&+{c_p}\int_{Q^{\tau}}\eta^2\chi(t)\cdot g\left(\left|Du^\varepsilon\right|^2\right){\left|D^2u^\varepsilon\right|^2\frac{\left(\left|Du^\varepsilon\right|-1\right)^p}{\left|Du^\varepsilon\right|^2}}\,dz\cr\cr
						&&+c_p\varepsilon\int_{{Q^{\tau}}}\eta^2\chi(t)g\left(\left|Du^\varepsilon\right|^2\right)\left|DV_{p}\left(Du^\varepsilon\right)\right|^2\,dz\cr\cr
						&\le&2\int_{Q^{\tau}}\eta\chi(t)\left|D\eta\right|\left|DH_{p-1}\left(Du^\varepsilon\right)\right|\left|Du^\varepsilon\right| \cdot g\left(\left|Du^\varepsilon\right|^2\right)\,dz\cr\cr
						&&+2\varepsilon	\int_{Q^{\tau}}\eta\chi(t)\left|D\eta\right|\left|D\left(\left(1+\left|Du^\varepsilon\right|^2\right)^{\frac{p-2}{2}}Du^\varepsilon\right)\right|\left|Du^\varepsilon\right| \cdot g\left(\left|Du^\varepsilon\right|^2\right)\,dz\cr\cr
						&&+\int_{Q^{\tau}}\chi(t)\left|f^\varepsilon\right|\left|D\left(\eta^2Du^\varepsilon\cdot g\left(\left|Du^\varepsilon\right|^2\right)\right)\right|\,dz\cr\cr
						&&+\frac{1}{2}\int_{Q^{\tau}}\eta^2\partial_t\chi(t)\left(\int_0^{\left|Du^\varepsilon\right|^2}g(s)\,ds\right)\,dz\cr\cr
						&=:&\tilde{I}_1+\tilde{I}_2+\tilde{I}_3+\tilde{I}_4,
					\end{eqnarray}
					for every $g\in W^{1,\infty}(0,+\infty)$ such that \eqref{**} holds true.
					Now, by \eqref{DH_p-1} and by Young's inequality,  we have  
					\begin{eqnarray}\label{tilde{I}_1}
						\tilde{I}_1+\tilde{I}_2&\le&c_p\int_{Q^{\tau}}\eta\chi(t)\left|D\eta\right|\left(\left|Du^\varepsilon\right|-1\right)_+^{p-2}\left|D^2u^\varepsilon\right|\left|Du^\varepsilon\right| \cdot g\left(\left|Du^\varepsilon\right|^2\right)\,dz\cr\cr
						&&+c_p\cdot\varepsilon\int_{Q^{\tau}}\eta\chi(t)\left|D\eta\right|\left(1+\left|Du^\varepsilon\right|^2\right)^{\frac{p-1}{2}}\left|D^2u^\varepsilon\right| \cdot g\left(\left|Du^\varepsilon\right|^2\right)\,dz\cr\cr
						&\le&\sigma\int_{Q^{\tau}}\eta^2\chi(t){\frac{\left(\left|Du^\varepsilon\right|-1\right)_+^p}{\left|Du^\varepsilon\right|^2}}\left|D^2u^\varepsilon\right|^2\cdot g\left(\left|Du^\varepsilon\right|^2\right)\,dz\cr\cr
						&&+\sigma\varepsilon \int_{Q^{\tau}}\eta^2\chi(t)\left(1+\left|Du^\varepsilon\right|^2\right)^{\frac{p-2}{2}}\left|D^2u^\varepsilon\right|^2\cdot g\left(\left|Du^\varepsilon\right|^2\right)\,dz\cr\cr
						&&+c_\sigma\int_{Q^{\tau}}\chi(t)\left|D\eta\right|^2{\left(\left|Du^\varepsilon\right|-1\right)_+^{p-4}\left|Du^\varepsilon\right|^4}\cdot g\left(\left|Du^\varepsilon\right|^2\right)\,dz\cr\cr
						&&+c_{p,\sigma}\cdot \varepsilon\int_{Q^{\tau}}\chi(t)\left|D\eta\right|^2\left(1+\left|Du^\varepsilon\right|^2\right)^{\frac{p}{2}}\cdot g\left(\left|Du^\varepsilon\right|^2\right)\,dz\cr\cr
						&\le&\sigma\int_{Q^{\tau}}\eta^2\chi(t){\frac{\left(\left|Du^\varepsilon\right|-1\right)_+^p}{\left|Du^\varepsilon\right|^2}}\left|D^2u^\varepsilon\right|^2\cdot g\left(\left|Du^\varepsilon\right|^2\right)\,dz\cr\cr
						&&+\sigma\varepsilon \int_{Q^{\tau}}\eta^2\chi(t)\left|DV_p\left(Du^\varepsilon\right)\right|^2\cdot g\left(\left|Du^\varepsilon\right|^2\right)\,dz\cr\cr
						&&+c_{\sigma,p}\left\Vert D\eta\right\Vert_{L^\infty}^2\left\Vert g\right\Vert_{L^\infty}\int_{Q^{\tau}}\chi(t)\left(1+\left|Du^\varepsilon\right|\right)^p\,dz,
					\end{eqnarray}
					where we used  \eqref{derV}, and where  $\sigma>0$ is a parameter that will be chosen later.\\
					Now, using Young's Inequality, we estimate the term $\tilde{I}_3$, as follows
					\begin{eqnarray}\label{tilde{I}_3}
						\tilde{I}_3&\le&c\int_{Q^{\tau}}\chi(t)\left|f^\varepsilon\right|\eta\left|D\eta\right|\left|Du^\varepsilon\right|\cdot g\left(\left|Du^\varepsilon\right|^2\right)\,dz\cr\cr
						&&+c\int_{Q^{\tau}}\chi(t)\left|f^\varepsilon\right|\eta^2\left|D^2u^\varepsilon\right|\cdot g\left(\left|Du^\varepsilon\right|^2\right)\,dz\cr\cr
						&&+c\int_{Q^{\tau}}\chi(t)\left|f^\varepsilon\right|\eta^2\left|Du^\varepsilon\right|^2\left|D^2u^\varepsilon\right|\cdot g'\left(\left|Du^\varepsilon\right|^2\right)\,dz\cr\cr
						&\le&c\left\Vert D\eta\right\Vert_\infty\left\Vert g\right\Vert_{L^\infty}\int_{Q^{\tau}}\eta\chi(t)\left|f^\varepsilon\right|^2\,dz\cr&&+ c\left\Vert D\eta\right\Vert_\infty\left\Vert g\right\Vert_{L^\infty}\int_{Q^\tau}\eta\chi(t)\left|Du^\varepsilon\right|^2\,dz\cr\cr
						&&+c\int_{Q^{\tau}}\eta^2\chi(t)\left|f^\varepsilon\right|\left|D^2u^\varepsilon\right|\cdot g\left(\left|Du^\varepsilon\right|^2\right)\,dz\cr\cr
						&&+c\int_{Q^{\tau}}\eta^2\chi(t)\left|f^\varepsilon\right|\left|Du^\varepsilon\right|^2\left|D^2u^\varepsilon\right|\cdot g'\left(\left|Du^\varepsilon\right|^2\right)\,dz.
					\end{eqnarray}
					Plugging \eqref{tilde{I}_1} and \eqref{tilde{I}_3} into \eqref{eq22lim}, we get
					\begin{eqnarray*}\label{eq22lim*}
						&&\frac{1}{2}\int_{B_{{2\rho}}}\eta^2\chi(\tau)\left(\int_0^{\left|Du^\varepsilon(x,\tau)\right|^2}g(s)\,ds\right)\,dx\cr\cr
						&&+c_p\int_{Q^{\tau}}\eta^2\chi(t)\cdot g\left(\left|Du^\varepsilon\right|^2\right){\frac{\left(\left|Du^\varepsilon\right|-1\right)_+^p}{\left|Du^\varepsilon\right|^2}}\left|D^2u^\varepsilon\right|^2\,dz\cr\cr		&&+c_p\varepsilon\int_{{Q^{\tau}}}\eta^2\chi(t)g\left(\left|Du^\varepsilon\right|^2\right)\left|DV_{p}\left(Du^\varepsilon\right)\right|^2\,dz\cr\cr
						&\le&\sigma\int_{Q^{\tau}}\eta^2\chi(t){\frac{\left(\left|Du^\varepsilon\right|-1\right)_+^p}{|Du^\varepsilon|^2}}\left|D^2u^\varepsilon\right|^2\cdot g\left(\left|Du^\varepsilon\right|^2\right)\,dz\cr\cr
						&&+\sigma\varepsilon \int_{Q^{\tau}}\eta^2\chi(t)\left|DV_p\left(Du^\varepsilon\right)\right|^2\cdot g\left(\left|Du^\varepsilon\right|^2\right)\,dz\cr\cr
						&&+c_{p,\sigma}\left\Vert D\eta\right\Vert_\infty\left\Vert g\right\Vert_{L^\infty}\int_{Q^{\tau}}\eta\chi(t)\left|f^\varepsilon\right|^2\,dz\cr&&+ c_{p,\sigma}\Vert D\eta\Vert_\infty\left\Vert g\right\Vert_{L^\infty}\int_{Q^\tau}\eta\chi(t)\left(1+|Du^\varepsilon|\right)^p\,dz\cr\cr
						&&+c\int_{Q^{\tau}}\eta^2\chi(t)\left|f^\varepsilon\right|\left|D^2u^\varepsilon\right|\cdot g\left(\left|Du^\varepsilon\right|^2\right)\,dz\cr\cr
						&&+c\int_{Q^{\tau}}\eta^2\chi(t)\left|f^\varepsilon\right|\left|Du^\varepsilon\right|^2\left|D^2u^\varepsilon\right|\cdot g'\left(\left|Du^\varepsilon\right|^2\right)\,dz\cr\cr
						&&+\frac{1}{2}\int_{Q^{\tau}}\eta^2\partial_t\chi(t)\left(\int_0^{\left|Du^\varepsilon\right|^2}g(s)\,ds\right)\,dz, 
					\end{eqnarray*}
					which, for a sufficiently small $\sigma$, gives
					\begin{eqnarray*}\label{eq22lim'**}
						&&\frac{1}{2}\int_{B_{{2\rho}}}\eta^2\chi(\tau)\left(\int_0^{\left|Du^\varepsilon(x,\tau)\right|^2}g(s)\,ds\right)\,dx\cr\cr
						&&+c_p\int_{Q^{\tau}}\eta^2\chi(t)\cdot g\left(\left|Du^\varepsilon\right|^2\right){\frac{\left(\left|Du^\varepsilon\right|-1\right)_+^p}{|Du^\varepsilon|^2}}\left|D^2u^\varepsilon\right|^2\,dz\cr\cr
						&&+c_p\varepsilon\int_{{Q^{\tau}}}\eta^2\chi(t)g\left(\left|Du^\varepsilon\right|^2\right)\left|DV_{p}\left(Du^\varepsilon\right)\right|^2\,dz\cr\cr
						&\le&c_{p}\Vert D\eta\Vert_\infty\left\Vert g\right\Vert_{L^\infty}\int_{Q^{\tau}}\eta\chi(t)\left|f^\varepsilon\right|^2\,dz\cr&&+ c_{p}\Vert D\eta\Vert_\infty\left\Vert g\right\Vert_{L^\infty}\int_{Q^\tau}\eta\chi(t)\left(1+|Du^\varepsilon|\right)^p\,dz\cr\cr
						&&+c\int_{Q^{\tau}}\eta^2\chi(t)\left|f^\varepsilon\right|\left|D^2u^\varepsilon\right|\cdot g\left(\left|Du^\varepsilon\right|^2\right)\,dz\cr\cr
						&&+c\int_{Q^{\tau}}\eta^2\chi(t)\left|f^\varepsilon\right|\left|Du^\varepsilon\right|^2\left|D^2u^\varepsilon\right|\cdot g'\left(\left|Du^\varepsilon\right|^2\right)\,dz\cr\cr
						&&+\frac{1}{2}\int_{Q^{\tau}}\eta^2\partial_t\chi(t)\left(\int_0^{\left|Du^\varepsilon\right|^2}g(s)\,ds\right)\,dz,
					\end{eqnarray*}
					that, neglecting the third integral in the left hand side, implies
					\begin{eqnarray}\label{eq22lim**}
						&&\frac{1}{2}\int_{B_{{2\rho}}}\eta^2\chi(\tau)\left(\int_0^{\left|Du^\varepsilon(x,\tau)\right|^2}g(s)\,ds\right)\,dx\cr\cr
						&&+c_p\int_{Q^{\tau}}\eta^2\chi(t)\cdot g\left(\left|Du^\varepsilon\right|^2\right){\frac{\left(\left|Du^\varepsilon\right|-1\right)_+^p}{\left|Du^\varepsilon\right|^2}}\left|D^2u^\varepsilon\right|^2\,dz\cr\cr
						&\le&c_{p}\Vert D\eta\Vert_\infty\left\Vert g\right\Vert_{L^\infty}\int_{Q^{\tau}}\eta\chi(t)\left|f^\varepsilon\right|^2\,dz\cr&&+ c_{p}\Vert D\eta\Vert_\infty\left\Vert g\right\Vert _{L^\infty}\int_{Q^\tau}\eta\chi(t)\left(1+|Du^\varepsilon|\right)^p\,dz\cr\cr
						&&+c\int_{Q^{\tau}}\eta^2\chi(t)\left|f^\varepsilon\right|\left|D^2u^\varepsilon\right|\cdot g\left(\left|Du^\varepsilon\right|^2\right)\,dz\cr\cr
						&&+c\int_{Q^{\tau}}\eta^2\chi(t)\left|f^\varepsilon\right|\left|Du^\varepsilon\right|^2\left|D^2u^\varepsilon\right|\cdot g'\left(\left|Du^\varepsilon\right|^2\right)\,dz\cr\cr
						&&+\frac{1}{2}\int_{Q^{\tau}}\eta^2\partial_t\chi(t)\left(\int_0^{\left|Du^\varepsilon\right|^2}g(s)\,ds\right)\,dz,
				\end{eqnarray}}

				Now, for $\delta\in (0,1)$, recalling the notation in \eqref{gkdef}, we choose 
				
				$$
				g(s)=g_{1+\delta}\left(\left(s-1-\delta\right)_+\right)
				$$
				
				that is
				$$
				g(s)=\frac{\left(s-1-\delta\right)_+^2}{1+\delta+\left(s-1-\delta\right)_+^2},
				$$
				
				that is legitimate since $g\in W^{1,\infty}([0,+\infty)).$\\
				{Moreover, with this choice, we have
					$g(s)\in[0,1]$, for every $s\ge0$,
					and thanks to \eqref{g_k'est}, there exists a constant $c_\delta>0$ such that
					$$sg'\left(s^2\right)\le c_\delta\qquad\mbox{ for every }s\ge0,$$
					so that \eqref{**} holds.}
				Therefore, since $g(s)$ vanishes on the set where $s\le 1+\delta$  and $g(s)\le 1$ for every $s$, \eqref{eq22lim**} becomes
				\begin{eqnarray*}\label{eq22ter}
					&&\frac{1}{2}\int_{B_{{2\rho}}}\eta^2\chi(\tau)\left(\int_0^{\left|Du^\varepsilon(x,\tau)\right|^2}g(s)\,ds\right)\,dx\cr\cr
					&&+c_p\int_{Q^{\tau}}\eta^2\chi(t)\cdot g\left(\left|D u^\varepsilon\right|^2\right){\frac{\left(\left|Du^\varepsilon\right|-1\right)_+^p}{|Du^\varepsilon|^2}}\left|D^2u^\varepsilon\right|^2\,dz\cr\cr
					&\le& c\int_{Q^{\tau}\cap \{{\left|Du^\varepsilon\right|^2}>1+\delta\}}\eta^2\chi(t)\left|f^\varepsilon\right|\left|D^2 u^\varepsilon\right|{\frac{\left(\left|Du^\varepsilon\right|-1\right)_+^{\frac{p}{2}}}{\left|Du^\varepsilon\right|}}{\frac{\left|Du^\varepsilon\right|}{\left(\left|Du^\varepsilon\right|-1\right)_+^{\frac{p}{2}}}}\cdot g\left(\left|D u^\varepsilon\right|^2\right)\,dz\cr\cr
					&&+c\int_{Q^{\tau}\cap \{{\left|Du^\varepsilon\right|^2}>1+\delta\}}\eta^2\chi(t)\left|f^\varepsilon\right|\left|D u^\varepsilon\right|^2{\frac{\left(\left|Du^\varepsilon\right|-1\right)_+^{\frac{p}{2}}}{\left|Du^\varepsilon\right|}}{\frac{\left|Du^\varepsilon\right|}{\left(\left|Du^\varepsilon\right|-1\right)_+^{\frac{p}{2}}}}\left|D^2 u^\varepsilon\right|
					g'\left(\left|Du^\varepsilon\right|^2\right)\,dz\cr\cr
					&&+ c_{p}\Vert D\eta\Vert_\infty\left\Vert\chi\right\Vert_{L^\infty}\int_{Q^\tau}\left(1+|Du^\varepsilon|^p+\left|f^\varepsilon\right|^2\right)\,dz+\int_{Q^{\tau}}\eta^2\partial_t\chi(t)\left(\int_0^{\left|Du^\varepsilon\right|^2}g(s)\,ds\right)\,dz\cr\cr
					&\le& {\frac{c_p}{\delta^{\frac{p}{2}}}}\int_{Q^{\tau}}\eta^2\chi(t)\left|f^\varepsilon\right|\left|D^2 u^\varepsilon\right|{\frac{\left(\left|Du^\varepsilon\right|-1\right)_+^{\frac{p}{2}}}{\left|Du^\varepsilon\right|}}\cdot g\left(\left|D u^\varepsilon\right|^2\right)\,dz\cr\cr
					&&+{\frac{c_p}{\delta^{\frac{p}{2}}}}\int_{Q^{\tau}}\eta^2\chi(t)\left|f^\varepsilon\right|\left|D u^\varepsilon\right|^2{\frac{\left(\left|Du^\varepsilon\right|-1\right)_+^{\frac{p}{2}}}{\left|Du^\varepsilon\right|}}\left|D^2 u^\varepsilon \right|
					g'\left(\left|Du^\varepsilon\right|^2\right)\,dz\cr\cr
					&&+ c_{p}\Vert D\eta\Vert_\infty\left\Vert\chi\right\Vert_{L^\infty}\int_{Q^\tau}\left(1+|Du^\varepsilon|^p+\left|f^\varepsilon\right|^2\right)\,dz+\int_{Q^{\tau}}\eta^2\partial_t\chi(t)\left(\int_0^{\left|Du^\varepsilon\right|^2}g(s)\,ds\right)\,dz,
				\end{eqnarray*}
				{where we used that  $$\sup_{x\in \left(\sqrt{1+\delta},+\infty\right)}\frac{x}{(x-1)^{\frac{p}{2}}}=\frac{\sqrt{1+\delta}}{\left(\sqrt{1+\delta}-1\right)^{\frac{p}{2}}}=\frac{\sqrt{1+\delta}\left(\sqrt{1+\delta}+1\right)^{\frac{p}{2}}}{\delta^{\frac{p}{2}}}\le \frac{c_p}{\delta^{\frac{p}{2}}},$$}
				since $\delta<1$. Using   Young's inequality in the first integral in the right hand, previous estimate yields
				\begin{eqnarray*}\label{eq22quater}
					&&\frac12\int_{B_{{2\rho}}}\eta^2\chi(\tau)\left(\int_0^{\left|Du^\varepsilon(x,\tau)\right|^2}g(s)\,ds\right)\,dx\cr\cr
					&&+c_p\int_{Q^{\tau}}\eta^2\chi(t)\cdot g\left(\left|Du^\varepsilon\right|^2\right){\frac{\left(\left|Du^\varepsilon\right|-1\right)_+^p}{\left|Du^\varepsilon\right|^2}}\left|D^2u^\varepsilon\right|^2\,dz\cr\cr
					&\le& {\frac{c_p(\beta)}{\delta^p}}\int_{Q^{\tau}}\eta^2\chi(t)\left|f^\varepsilon\right|^2\cdot g\left(\left|D u^\varepsilon\right|^2\right)\,dz\cr\cr
					&&+\beta\int_{Q^{\tau}}\eta^2\chi(t){\frac{\left(\left|Du^\varepsilon\right|-1\right)_+^p}{\left|Du^\varepsilon\right|^2}}\left|D^2u^\varepsilon\right|^2\cdot g\left(\left|D u^\varepsilon\right|^2\right)\,dz\cr\cr
					&&+{\frac{c_p}{\delta^{\frac{p}{2}}}}\int_{Q^{\tau}}\eta^2\chi(t)\left|f^\varepsilon\right|\left|D u^\varepsilon\right|^2{\frac{\left(\left|Du^\varepsilon\right|-1\right)_+^{\frac{p}{2}}}{\left|Du^\varepsilon\right|}}\left|D^2 u^\varepsilon \right|
					g'\left(\left|Du^\varepsilon\right|^2\right)\,dz\cr\cr
					&&+ c_{p}\Vert D\eta\Vert_\infty\left\Vert\chi\right\Vert_{L^\infty}\int_{Q^\tau}\left(1+|Du^\varepsilon|^p+\left|f^\varepsilon\right|^2\right)\,dz\cr\cr
					&&+\int_{Q^{\tau}}\eta^2\partial_t\chi(t)\left(\int_0^{\left|Du^\varepsilon\right|^2}g(s)\,ds\right)\,dz.
				\end{eqnarray*}
				Choosing $\beta$ sufficiently small, reabsorbing the second integral in the right hand side by the left hand side and using that $g(s)\le 1$, we get
				\begin{eqnarray}\label{eq225}
					&&\int_{B_{{2\rho}}}\eta^2\chi(\tau)\left(\int_0^{\left|Du^\varepsilon(x,\tau)\right|^2}g(s)\,ds\right)\,dx\cr\cr
					&&+\int_{Q^{\tau}}\eta^2\chi(t)\cdot g\left(\left|D u^\varepsilon\right|^2\right){\frac{\left(\left|Du^\varepsilon\right|-1\right)_+^p}{\left|Du^\varepsilon\right|^2}}\left|D^2u^\varepsilon\right|^2\,dz\cr\cr
					&\le& c{\frac{c_p}{\delta^{\frac{p}{2}}}}\int_{Q^{\tau}}\eta^2\chi(t)\left|f^\varepsilon\right|\left|D u^\varepsilon\right|^2{\frac{\left(\left|Du^\varepsilon\right|-1\right)_+^{\frac{p}{2}}}{\left|Du^\varepsilon\right|}}\left|D^2 u^\varepsilon \right|
					g'\left(\left|Du^\varepsilon\right|^2\right)\,dz\cr\cr
					&&+\int_{Q^{\tau}}\eta^2\partial_t\chi(t)\left(\int_0^{\left|Du^\varepsilon\right|^2}g(s)\,ds\right)\,dz\cr\cr
					&&c\left\Vert D\eta\right\Vert_{\infty}^2\left\Vert\chi\right\Vert_{\infty}\int_{Q^{\tau}}\left(1+\left|D u^\varepsilon\right|\right)^p\,dz\cr\cr
					&&+c\left\Vert\chi\right\Vert_{L^\infty}\left({\frac{c_p}{\delta^p}}+\left\Vert D\eta\right\Vert_{L^\infty}\right)\int_{Q^{\tau}}\left|f^\varepsilon\right|^2\,dz.
				\end{eqnarray}
				We now estimate the first integral in the right side of previous inequality with the use of \eqref{propgfinale} with $s=\left|Du^\varepsilon\right|^2$, $A={\displaystyle\frac{\left(\left|Du^\varepsilon\right|-1\right)_+^{\frac{p}{2}}}{\left|Du^\varepsilon\right|}}\left|D^2 u^\varepsilon \right|$, $B={\displaystyle\frac{c_p}{\delta^{\frac{p}{2}}}}\left|f^\varepsilon\right|$ and $k=1+\delta$, thus getting
				\begin{eqnarray*}\label{Ii}
					&&{\frac{c_p}{\delta^{\frac{p}{2}}}}\int_{Q^{\tau}}\eta^2\chi(t)\left|f^\varepsilon\right|\left|D u^\varepsilon\right|^2{\frac{\left(\left|Du^\varepsilon\right|-1\right)_+^{\frac{p}{2}}}{\left|Du^\varepsilon\right|}}\left|D^2 u^\varepsilon \right|
					g'\left(\left|Du^\varepsilon\right|^2\right)\,dz\cr\cr
					&\le&{2}\alpha \int_{Q^{\tau}}\eta^2\chi(t){\frac{\left(\left|Du^\varepsilon\right|-1\right)_+^{p}}{\left|Du^\varepsilon\right|^2}}\left|D^2 u^\varepsilon \right|^2
					g\left(\left|Du^\varepsilon\right|^2\right)\,dz\cr\cr
					&&+{2}\alpha\sigma\int_{Q^{\tau}}\eta^2\chi(t){\frac{\left(\left|Du^\varepsilon\right|-1\right)_+^{p}}{\left|Du^\varepsilon\right|^2}}\left|D^2 u^\varepsilon \right|^2\,dz\cr\cr
					&&+{\frac{c_{\alpha,p}}{\delta^p}}\int_{Q^{\tau}}\eta^2\chi(t)\left|f^\varepsilon\right|^2\,dz,
				\end{eqnarray*}
				with  constants $c,c_\alpha$ both independent of $\sigma$ {and where we used that $\delta<1$}. By virtue of \eqref{derH},  taking the limit as $\sigma\to 0$ in previous inequality, we have
				\begin{eqnarray}\label{Iii}
					&&{\frac{c_p}{\delta^{\frac{p}{2}}}}\int_{Q^{\tau}}\eta^2\chi(t)\left|f^\varepsilon\right|\left|D u^\varepsilon\right|^2{\frac{\left(\left|Du^\varepsilon\right|-1\right)_+^{\frac{p}{2}}}{\left|Du^\varepsilon\right|}}\left|D^2 u^\varepsilon \right|
					g'\left(\left|Du^\varepsilon\right|^2\right)\,dz\cr\cr
					&\le&  {2}\alpha \int_{Q^{\tau}}\eta^2\chi(t){\frac{\left(\left|Du^\varepsilon\right|-1\right)_+^{p}}{\left|Du^\varepsilon\right|^2}}\left|D^2 u^\varepsilon \right|^2
					g\left(\left|Du^\varepsilon\right|^2\right)\,dz\cr\cr
					&&+{\frac{c_{\alpha,p}}{\delta^p}}\int_{Q^{\tau}}\eta^2\chi(t)\left|f^\varepsilon\right|^2\,dz,
				\end{eqnarray}
				Inserting \eqref{Iii}  in \eqref{eq225}, we find
				\begin{eqnarray*}\label{eq226*}
					&&\int_{B_{{2\rho}}}\eta^2\chi(\tau)\left(\int_0^{\left|Du^\varepsilon(x,\tau)\right|^2}g(s)\,ds\right)\,dx\cr\cr
					&&+\int_{Q^{\tau}}\eta^2\chi(t)\cdot g\left(\left|D u^\varepsilon\right|^2\right){\frac{\left(\left|Du^\varepsilon\right|-1\right)_+^{p}}{\left|Du^\varepsilon\right|^2}}\left|D^2 u^\varepsilon \right|^2\,dz\cr\cr
					&\le&{2}\alpha \int_{Q^{\tau}}\eta^2\chi(t){\frac{\left(\left|Du^\varepsilon\right|-1\right)_+^{p}}{\left|Du^\varepsilon\right|^2}}\left|D^2 u^\varepsilon \right|^2
					g\left(\left|Du^\varepsilon\right|^2\right)\,dz\cr\cr
					&&+{\frac{c_{\alpha,p}}{\delta^p}}\int_{Q^{\tau}}\eta^2\chi(t)|f^\varepsilon|^2\,dz\cr\cr
					&&+\int_{Q^{\tau}}\eta^2\partial_t\chi(t)\left(\int_0^{\left|Du^\varepsilon\right|^2}g(s)\,ds\right)\,dz\cr\cr
					&&c\left\Vert D\eta\right\Vert_{\infty}^2\left\Vert\chi\right\Vert_{\infty}\int_{Q^{\tau}}\left(1+|D u^\varepsilon|\right)^p\,dz\cr\cr
					&&+c\left\Vert\chi\right\Vert_{L^\infty}\left({\frac{c_p}{\delta^p}}+\left\Vert D\eta\right\Vert_{L^\infty}\right)\int_{Q^{\tau}}\left|f^\varepsilon\right|^2\,dz.
				\end{eqnarray*}
				Choosing $\alpha=\displaystyle\frac{1}{4}$ , we can reabsorb the first integral in the right hand side by the left hand side, thus obtaining
				\begin{eqnarray}\label{eq226***}
					&&\int_{B_{{2\rho}}}\eta^2\chi(\tau)\left(\int_0^{\left|Du^\varepsilon(x,\tau)\right|^2}g(s)\,ds\right)\,dx\cr\cr
					&&+\int_{Q^{\tau}}\eta^2\chi(t)\cdot g\left(\left|D u^\varepsilon\right|^2\right){\frac{\left(\left|Du^\varepsilon\right|-1\right)_+^{p}}{|Du^\varepsilon|^2}}\left|D^2 u^\varepsilon \right|^2\,dz\cr\cr
					&\le& c\left\Vert D\eta\right\Vert_{\infty}^2\left\Vert\chi\right\Vert_{\infty}\int_{Q^{\tau}}\left(1+\left|D u^\varepsilon\right|\right)^p\,dz\cr\cr
					&&+{\frac{c}{\delta^p}}\left\Vert\chi\right\Vert_{L^\infty}\left(1+\left\Vert D\eta\right\Vert_{L^\infty}\right)\int_{Q^{\tau}}\left|f^\varepsilon\right|^2\,dz\cr\cr 
					&&+c\int_{Q^{\tau}}\eta^2\partial_t\chi(t)\left(\int_0^{\left|Du^\varepsilon\right|^2}g(s)\,ds\right)\,dz.
				\end{eqnarray}
				By the definition of $g$, we have
				$$
				\int_0^\zeta g(s)\,ds=
				\begin{cases}
					0\qquad&\mbox{ if }\quad 0<\zeta\le1+\delta\vspace{11pt}\\
					\displaystyle\int_{1+\delta}^\zeta\frac{\left(s-1-\delta\right)^2}{1+\delta+\left(s-1-\delta\right)^2}\,ds\qquad&\mbox{ if }\quad\zeta>1+\delta,
				\end{cases}
				$$
				
				and so it is easy to check that
				
				$$
				\int_0^\zeta g(s)\,ds=\begin{cases}
					0\qquad&\mbox{ if }\quad 0<\zeta\le1+\delta\vspace{11pt}\\
					\zeta-1-\delta-\sqrt{1+\delta}\arctan\displaystyle\left(\frac{\zeta-1-\delta}{\sqrt{1+\delta}}\right)\qquad&\mbox{ if }\quad \zeta>1+\delta,
				\end{cases}
				$$
				
				that is 
				
				\begin{equation*}\label{intg}
					\int_0^\zeta g(s)\,ds=
					\left(\zeta-1-\delta\right)_+-\sqrt{1+\delta}\arctan\left[\frac{\left(\zeta-1-\delta\right)_+}{\sqrt{1+\delta}}\right]. 
				\end{equation*}
				Therefore, by previous equality and the properties of $\chi$ and $\eta$, \eqref{eq226***} implies
				\begin{eqnarray}\label{eq227}
					&&\int_{B_{{2\rho}}}\eta^2\chi(\tau)\left(\left|Du^\varepsilon(x,\tau)\right|^2-1-\delta\right)_+\,dx\cr\cr
					&&+\int_{Q^{\tau}}\eta^2\chi(t)\cdot g\left(\left|D u^\varepsilon\right|^2\right){\frac{\left(\left|Du^\varepsilon\right|-1\right)_+^{p}}{|Du^\varepsilon|^2}}\left|D^2 u^\varepsilon \right|^2\,dz\cr\cr
					&\le&c\left\Vert D\eta\right\Vert_{\infty}^2\left\Vert\chi\right\Vert_{\infty}\int_{Q^{\tau}}\left(1+\left|D u^\varepsilon\right|\right)^p\,dz\cr\cr
					&&+{\frac{c}{\delta^p}}\left\Vert\chi\right\Vert_{L^\infty}\left(1+\left\Vert D\eta\right\Vert_{L^\infty}\right)\int_{Q^{\tau}}\left|f^\varepsilon\right|^2\,dz\cr\cr 
					&&+c\int_{Q^\tau}\eta^2\partial_t\chi(t)\left(\left|Du^\varepsilon\right|^2-1-\delta\right)_+\,dz\cr\cr
					&&+c\left\Vert \partial_t\chi\right\Vert_{\infty}\left|Q^\tau\right|+c\left\Vert\chi\right\Vert_\infty\left|B_R\right|,
				\end{eqnarray}
				which holds for almost every $\tau\in\left(t_0-{4\rho^2}, t_0\right)$.\\ 
				We now choose a cut-off function $\eta\in C^{\infty}\left(B_{{2\rho}}\left(x_0\right)\right)$ with $\eta\equiv1$ on $B_\rho\left(x_0\right)$ such that $0\le\eta\le1$ and $\left|D\eta\right|\le\displaystyle\frac{c}{{\rho}}.$ For the cut-off function in time, we choose $\chi\in W^{1, \infty}\left(t_0-{R^2}, t_0, \left[0, 1\right]\right)$ such that $\chi\equiv0$ on $\left(t_0-{R^2}, t_0-{4\rho^2}\right]$, $\chi\equiv1$ on $\left[t_0-\rho^2, t_0\right)$ and $\partial_t\chi\le\displaystyle\frac{c}{{\rho^2}}$ on $\left(t_0-{4\rho^2}, t_0-\rho^2\right)$.
				With these choices, \eqref{eq227} gives
				
				\begin{eqnarray*}\label{eq228}
					&&\sup_{\tau\in\left(t_0-{4\rho^2}, t_0\right)}\int_{B_{\rho}}\chi(\tau)\left(\left|Du^\varepsilon(x,\tau)\right|^2-1-\delta\right)_+\,dx\cr\cr
					&&+\int_{Q_\rho}g\left(\left|D u^\varepsilon\right|^2\right){\frac{\left(\left|Du^\varepsilon\right|-1\right)_+^{p}}{\left|Du^\varepsilon\right|^2}}\left|D^2 u^\varepsilon \right|^2\,dz\cr\cr
					&\le&\frac{c}{{\rho^2}}\int_{Q_{2\rho}}\left(1+\left|D u^\varepsilon\right|^p\right)\,dz+\frac{c}{\rho^2\delta^p}\int_{Q_{2\rho}}\left|f^\varepsilon\right|^2 \,dz\cr\cr
					&&+\frac{c\left|Q_{2\rho}\right|}{\rho^2}+c\left|B_{2\rho}\right|,
				\end{eqnarray*}
				and since $\rho<2\rho<R<1$,  and $Q_{{2\rho}}=B_{\rho}\times\left(t_0-{4\rho^2}, t_0\right)$, we have
				\begin{eqnarray}\label{eq229*}
					&&\sup_{\tau\in\left(t_0-{4\rho}^2, t_0\right)}\int_{B_{\rho}}\left(\left|Du^\varepsilon(x,\tau)\right|^2-1-\delta\right)_+\,dx\cr\cr
					&&+\int_{Q_{\rho}}g\left(\left|D u^\varepsilon\right|^2\right){\frac{\left(\left|Du^\varepsilon\right|-1\right)_+^{p}}{\left|Du^\varepsilon\right|^2}}\left|D^2 u^\varepsilon \right|^2\,dz\cr\cr
					&\le&\frac{c}{\rho^2}\int_{Q_{2\rho}}\left(1+\left|D u^\varepsilon\right|^p\right)\,dz+\frac{c}{\rho^2\delta^p}\int_{Q_{2\rho}}\left|f^\varepsilon\right|^2 \,dz.
				\end{eqnarray}
				
				Now, with $\mathcal{G}_\delta(t)$ defined at \eqref{Gdef}, recalling \eqref{G'}, we have
				\begin{eqnarray*}
					&&\left|D\left[\mathcal{G}_\delta\left(\left(\left|Du^\varepsilon\right|-\delta -1\right)_+\right)\right]\right|^2\cr\cr
					&\le&  \frac{\left(\left|Du^\varepsilon\right|-\delta -1\right)_+^2}{1+\delta+\left(\left|Du^\varepsilon\right|-\delta -1\right)_+^2}\left[\left(\left|Du^\varepsilon\right|-\delta -1\right)_++\delta\right]^{p-2}\left|D^2u^\eps\right|^2\cr\cr 
					&=& g\left(\left|Du^\eps\right|\right)\left[\left(\left|Du^\varepsilon\right|-\delta -1\right)_++\delta\right]^{p-2}\left|D^2u^\eps\right|^2.
				\end{eqnarray*}
				Since $g(s)$ is nondecreasing, we have $g(s)\le g\left(s^2\right)$, and therefore
				\begin{eqnarray}\label{derG}
					&&\left|D\left[\mathcal{G}_\delta\left(\left(\left|Du^\varepsilon\right|-\delta -1\right)_+\right)\right]\right|^2\le g\left(\left|Du^\eps\right|^2\right)   \left(\left|Du^\varepsilon\right|-1\right)_+^{p-2}\left|D^2u^\eps\right|^2\cr\cr 
					&\le& {\frac{c_p}{\delta^2}} g\left(\left|Du^\eps\right|^2\right){\frac{\left(\left|Du^\varepsilon\right|-1\right)_+^{p}}{\left|Du^\varepsilon\right|^2}}\left|D^2 u^\varepsilon \right|^2,
				\end{eqnarray}
				where we also used that $g(s)=0$, for $0<s\le 1+\delta.$
				Using \eqref{derG} in the left hand side of \eqref{eq229*}, we obtain
				\begin{eqnarray*}
					&&\sup_{\tau\in\left(t_0-4\rho^2, t_0\right)}\int_{B_{\rho}}\left(\left|Du^\varepsilon(x,\tau)\right|^2-1-\delta\right)_+\,dx\cr\cr
					&&+\int_{Q_\rho}\left|D\left[\mathcal{G}_\delta\left(\left(\left|Du^\varepsilon\right|-\delta -1\right)_+\right)\right]\right|^2\,dz\cr\cr
					&\le&\frac{c}{\rho^2{\delta^2}}\left[\int_{Q_{2\rho}}\left(1+\left|D u^\varepsilon\right|^p\right)\,dz +\frac{1}{\delta^{p}}\int_{Q_{2\rho}}\left|f^\varepsilon\right|^2 \,dz\right],
				\end{eqnarray*}
				which is \eqref{uniformest}.
			\end{proof}
			Combining Lemma \ref{uniformestlemma} and Lemma \ref{lem:Giusti1}, we have the following.
			\begin{corollary}
				Let  $u^\varepsilon\in C^0\left(t_0-R^2,t_0;L^2\left(B_R\right)\right)\cap L^p\left(t_0-R^2,t_0; u+W_0^{1,p}\left(B_{R}\right)\right)$ be the unique solution to \eqref{app}. Then the following estimate
				\begin{eqnarray}\label{eq229t}
					&&\int_{Q_\frac{\rho}{2}}\left|\tau_h\left[\mathcal{G}_\delta\left(\left(\left|Du^\varepsilon\right|-\delta -1\right)_+\right)\right]\right|^2\,dz\cr\cr
					&\le&\frac{c|h|^2}{\rho^2{\delta^2}}\left[\int_{Q_{2\rho}}\left(1+\left|D u^\varepsilon\right|^p\right)\,dz +\frac{1}{\delta^{p}}\int_{Q_{2\rho}}\left|f^\varepsilon\right|^2\,dz\right]
				\end{eqnarray}
				holds for $\left|h\right|<\displaystyle\frac{\rho}{4}$, for any parabolic cylinder $Q_{2\rho}\Subset Q_R\left(z_0\right)$.
			\end{corollary}
			\section{Proof of Theorem \ref{Theorem1}}
			This section is devoted to the proof of Theorem \ref{Theorem1}, that will be divided in two steps.\\
			In the first one we shall establish an estimate that will allow us to measure the $L^2$-distance between $H_{\frac{p}{2}}\left(Du\right)$ and $H_{\frac{p}{2}}\left(Du^\eps\right)$ in terms of the $L^2$-distance between $f$ and $f^\varepsilon$.\\
			In the second one, we conclude combining this comparison estimate with the one obtained for the difference quotient of the solution to the regularized problem at \eqref{eq229t}.
			\begin{proof}[Proof of Theorem \ref{Theorem1}]
				{\bf Step 1: the comparison estimate.}\\
				We formally proceed by testing equations (\ref{eq1*}) and (\ref{app}) with the map $\varphi=k(t)\left(u^\varepsilon-u\right)$,
				where $k\in W^{1,\infty}\left(\mathbb{R}\right)$ is chosen such that\\
				\[
				k(t)=\begin{cases}
					\quad1\qquad&\mbox{ if }\quad t\leq t_{2},\vspace{11pt}\\
					-\displaystyle\frac{1}{\omega}\left(t-t_{2}-\omega\right)\qquad&\mbox{ if }\quad t_{2}<t<t_{2}+\omega,\vspace{11pt}\\
					\quad0\qquad&\mbox{ if }\quad t\geq t_{2}+\omega,\quad\quad
				\end{cases}
				\]
				\\
				with $t_{0}-{R^{2}}<t_{2}<t_{2}\,+\,\omega<t_{0}$, and then letting
				$\omega\rightarrow0$. We observe that, at this stage, it is important
				that $u^\varepsilon$ and $u$ agree on the parabolic boundary $\partial_{\mathrm{par}}Q_{{R}}\left(z_{0}\right)$.\\
				Proceeding in a standard way (see for example \cite{DMS}), for almost every $t_{2}\in\left(t_{0}\,-\,{R^{2}},t_{0}\right)$,
				we find
				
				\begin{eqnarray}\label{eq:ee11}
					&&\frac{1}{2}\int_{{B_{R}}\left(x_{0}\right)}\left|u^\varepsilon\left(x,t_{2}\right)-u\left(x,t_{2}\right)\right|^{2}\,dx\cr\cr
					&&+\int_{Q_{{R},t_{2}}}\left\langle H_{p-1}\left(Du^\varepsilon\right)-H_{p-1}\left(Du\right),Du^\varepsilon-Du\right\rangle\,dz\cr\cr
					&&+\varepsilon\int_{Q_{{R},t_{2}}}\left\langle\left(1+\left|Du^\varepsilon\right|^2\right)^\frac{p-2}{2}Du^\varepsilon,Du^\varepsilon-Du\right\rangle\,dz\cr\cr
					&=&\int_{Q_{{R},t_{2}}}\left(f-f^\varepsilon\right)\left(u^\varepsilon-u\right)\,dz,
				\end{eqnarray}
				
				where we used the abbreviation $Q_{{R},t_{2}}=B_{{R}}\left(x_{0}\right)\times\left(t_{0}-{R^{2}},t_{2}\right)$.
				Using Lemma \ref{lem:Brasco}, the Cauchy-Schwarz inequality as well
				as Young's inequality, from (\ref{eq:ee11}) we infer
				\begin{eqnarray}\label{eq:ee12}
					&&\lambda_p\sup_{t\in\left(t_{0}-{R^{2}},t_{0}\right)}\left\Vert u^\varepsilon(\cdot,t)-u(\cdot,t)\right\Vert_{L^{2}\left(B_{{R}}\left(x_{0}\right)\right)}^{2}\cr\cr
					&&+\lambda_p\int_{Q_{{R}}}\left|H_{\frac{p}{2}}\left(Du^\varepsilon\right)-H_{\frac{p}{2}}\left(Du\right)\right|^{2}\,dz+\varepsilon\int_{Q_{{R}}\left(z_{0}\right)}\left|Du^\varepsilon\right|^{p}\,dz\cr\cr
					&\leq&\int_{Q_{{R}}}\left|f-f^\varepsilon\right|\left|u^\varepsilon-u\right|\,dz+\varepsilon\int_{Q_{{R}}}\left|Du^\varepsilon\right|^{p-1}\left|Du\right|\,dz\cr\cr
					&\leq&\int_{Q_{R}}\left|f-f^\varepsilon\right|\left|u^\varepsilon-u\right|\,dz+\varepsilon\cdot c_p\int_{Q_{R}}\left|Du\right|^{p}\,dz\cr\cr
					&&+\frac{1}{2}\cdot \varepsilon\int_{Q_{R}}\left|Du^\varepsilon\right|^{p}\,dz,
				\end{eqnarray}
				where we set $\lambda_p=\min\Set{\displaystyle\frac{1}{2}, \displaystyle\frac{4}{p^2}}.$ Reabsorbing the last integral in the right-hand side of \eqref{eq:ee12} by the left-hand side, we arrive at
				\begin{eqnarray}\label{eq:ee13}
					&&\sup_{t\in\left(t_{0}-{R}^{2},t_{0}\right)}\left\Vert u^\varepsilon(\cdot,t)-u(\cdot,t)\right\Vert_{L^{2}\left(B_{R}\left(x_{0}\right)\right)}^{2}\cr\cr
					&&+\int_{Q_{R}}\left|H_{\frac{p}{2}}\left(Du^\varepsilon\right)-H_{\frac{p}{2}}\left(Du\right)\right|^{2}\,dz+\frac{\varepsilon}{2\lambda_p}\int_{Q_{R}}\left|Du^\varepsilon\right|^{p}\,dz\cr\cr
					&\leq&\varepsilon\,c_p\int_{Q_{R}}\left|Du\right|^{p}\,dz+c_p\int_{Q_{R}}\left|f-f^\varepsilon\right|\left|u^\varepsilon-u\right|\,dz.
				\end{eqnarray}
				{Using in turn H\"older's inequality 
					and Lemma \ref{lem:inter}, 
					we get
					\begin{eqnarray}\label{eq:ee14}
						\tilde{I}&:=&\int_{Q_{R}}\left|f-f^{\varepsilon}\right|\left|u^{\varepsilon}-u\right|\,dz\cr\cr
						&\le&C\left({R},n,p\right)\left\Vert f-f^{\varepsilon}\right\Vert_{L^2\left(Q_{R}\right)}\cdot\left(\int_{Q_{R}}\left|u^{\varepsilon}-u\right|^{p+\frac{2p}{n}}\,dz\right)^\frac{n}{p(n+2)}\cr\cr
							&\leq&c\left(n,p,{R}\right)\left\Vert f-f^{\varepsilon}\right\Vert_{L^2\left(Q_{R}\right)}\cdot\left(\int_{Q_{R}}\left|Du^{\varepsilon}-Du\right|^{p}\,dz\right)^{\frac{n}{p\left(n+2\right)}}\cr\cr
							&&\cdot\left(\sup_{t\in\left(t_{0}-{R^{2}},t_{0}\right)}\left\Vert u^{\varepsilon}(\cdot,t)-u(\cdot,t)\right\Vert_{L^{2}\left(B_{R}\left(x_{0}\right)\right)}^{2}\right)^{\frac{1}{n+2}}
						\end{eqnarray}
						Now, let us notice that
						\begin{eqnarray}\label{eq:ee16}
							&&\int_{Q_{R}}\left|Du^{\varepsilon}-Du\right|^{p}\,dz\cr\cr
							&=&\int_{Q_{R}\cap\left\{ \left|Du^{\varepsilon}\right|\geq1\right\} }\left(\left|Du^{\varepsilon}\right|-1+1\right)^{p}\,dz+\int_{Q_{R}\cap\left\{ \left|Du^{\varepsilon}\right|<1\right\} }\left|Du^{\varepsilon}\right|^{p}\,dz+\int_{Q_{R}}\left|Du\right|^{p}\,dz\cr\cr
							&\leq&c_p\int_{Q_{R}}\left[\left(\left|Du^{\varepsilon}\right|-1\right)_{+}^{p}\right]\,dz+\int_{Q_{R}}\left(\left|Du\right|^{p}+1\right)\,dz\cr\cr
							&\le&c_p\int_{Q_{R}}\left(\left|H_{\frac{p}{2}}\left(Du^{\varepsilon}\right)-H_{\frac{p}{2}}\left(Du\right)+H_{\frac{p}{2}}\left(Du\right)\right|^{2}\right)\,dz+c_p\int_{Q_{R}}\left(\left|Du\right|^{p}+1\right)\,dz\cr\cr
							&\leq&c_p\int_{Q_{R}}\left|H_{\frac{p}{2}}\left(Du^{\varepsilon}\right)-H_{\frac{p}{2}}\left(Du\right)\right|^{2}\,dz+c_p\int_{Q_{R}}\left(\left|Du\right|^{p}+1\right)\,dz.
						\end{eqnarray}
						Inserting \eqref{eq:ee16} in \eqref{eq:ee14}, we get
						\begin{eqnarray*}
							\tilde{I}&\leq&c\left(n,p,{R}\right)\left\Vert f-f^{\varepsilon}\right\Vert_{L^2\left(Q_{R}\left(z_0\right)\right)}\cr\cr
							&&\cdot\left(\int_{Q_{R}}\left|H_{\frac{p}{2}}\left(Du^{\varepsilon}\right)-H_{\frac{p}{2}}\left(Du\right)\right|^{2}\,dz+\int_{Q_{R}}\left(\left|Du\right|^{p}+1\right)\,dz\right)^{\frac{n}{p\left(n+2\right)}}\cr\cr
							&&\cdot\left(\sup_{t\in\left(t_{0}-{R^{2}},t_{0}\right)}\left\Vert u^{\varepsilon}(\cdot,t)-u(\cdot,t)\right\Vert_{L^{2}\left(B_{R}\left(x_{0}\right)\right)}^{2}\right)^{\frac{1}{n+2}}\cr\cr
							&\leq&c\left(n,p,{R}\right)\left\Vert f-f^{\varepsilon}\right\Vert_{L^2\left(Q_{R}\right)}\cdot\left(\int_{Q_{R}}\left|H_{\frac{p}{2}}\left(Du^{\varepsilon}\right)-H_{\frac{p}{2}}\left(Du\right)\right|^{2}\,dz\right)^{\frac{n}{p\left(n+2\right)}}\cr\cr
							&&\cdot\left(\sup_{t\in\left(t_{0}-{R^{2}},t_{0}\right)}\left\Vert u^{\varepsilon}(\cdot,t)-u(\cdot,t)\right\Vert_{L^{2}\left(B_{R}\left(x_{0}\right)\right)}^{2}\right)^{\frac{1}{n+2}}\cr\cr
							&&+c\left(n,p,{R}\right)\left\Vert f-f^{\varepsilon}\right\Vert_{L^2\left(Q_{R}\right)}\cdot\left(\int_{Q_{R}}\left(\left|Du\right|^{p}+1\right)\,dz\right)^{\frac{n}{p\left(n+2\right)}}\cr\cr
							&&\cdot\left(\sup_{t\in\left(t_{0}-{R^{2}},t_{0}\right)}\left\Vert u^{\varepsilon}(\cdot,t)-u(\cdot,t)\right\Vert_{L^{2}\left(B_{R}\left(x_{0}\right)\right)}^{2}\right)^{\frac{1}{n+2}}
						\end{eqnarray*}
						and, by Young's inequality, 
						we get 
						\begin{eqnarray}\label{eq:ee14*}
							\tilde{I}
							&\leq&\beta\int_{Q_{R}}\left|H_{\frac{p}{2}}\left(Du^{\varepsilon}\right)-H_{\frac{p}{2}}\left(Du\right)\right|^{2}\,dz+\beta\sup_{t\in\left(t_{0}-{R^{2}},t_{0}\right)}\left\Vert u^{\varepsilon}(\cdot,t)-u(\cdot,t)\right\Vert_{L^{2}\left(B_{R}\left(x_{0}\right)\right)}^{2}\cr\cr
							&&+c\left(n,p,R,\beta\right)\left\Vert f-f^{\varepsilon}\right\Vert_{L^2\left(Q_{R}\right)}^{\frac{n+2}{n+1}}\cdot\left(\int_{Q_{R}}\left(\left|Du\right|^{p}+1\right)\,dz\right)^{\frac{n}{p\left(n+1\right)}}\cr\cr
							&&+c\left(n, p, R, \beta\right)\left\Vert f-f^{\varepsilon}\right\Vert_{L^2\left(Q_{R}\right)}^\frac{p\left(n+2\right)}{n\left(p-1\right)+p}.
						\end{eqnarray}
						Inserting \eqref{eq:ee14*} in \eqref{eq:ee13}, we obtain
						\begin{eqnarray}\label{eq:ee15*}
							&&\sup_{t\in(t_{0}-{R^{2}},t_{0})}\left\Vert u^{\varepsilon}(\cdot,t)-u(\cdot,t)\right\Vert_{L^{2}\left(B_{R}\left(x_{0}\right)\right)}^{2}\cr\cr
							&&+\int_{Q_{R}}\left|H_{\frac{p}{2}}\left(Du^{\varepsilon}\right)-H_{\frac{p}{2}}\left(Du\right)\right|^{2}\,dz+\frac{\varepsilon}{2\lambda_p}\int_{Q_{R}}\left|Du^{\varepsilon}\right|^{p}\,dz\cr\cr
							&\leq&\beta\int_{Q_{R}}\left|H_{\frac{p}{2}}\left(Du^{\varepsilon}\right)-H_{\frac{p}{2}}\left(Du\right)\right|^{2}\,dz+\beta\sup_{t\in\left(t_{0}-{R^{2}},t_{0}\right)}\left\Vert u^{\varepsilon}(\cdot,t)-u(\cdot,t)\right\Vert_{L^{2}\left(B_{R}\left(x_{0}\right)\right)}^{2}\cr\cr
							&&+c\left(n,p,{R},\beta\right)\left\Vert f-f^{\varepsilon}\right\Vert_{L^2\left(Q_{ R}\right)}^{\frac{n+2}{n+1}}\cdot\left(\int_{Q_{R}}\left(\left|Du\right|^{p}+1\right)\,dz\right)^{\frac{n}{p\left(n+1\right)}}\cr\cr
							&&+c\left(n,p,{R},\beta\right)\left\Vert f-f^{\varepsilon}\right\Vert_{L^2\left(Q_{R}\right)}^\frac{p\left(n+2\right)}{n\left(p-1\right)+p}+\varepsilon\,c_p\int_{Q_{R}}\left|Du\right|^{p}\,dz.
						\end{eqnarray}
						Choosing $\displaystyle\beta=\frac{1}{2}$  and neglecting the third  non negative term  in the left hand side of \eqref{eq:ee15*}, we get
						\begin{eqnarray}\label{eq:ee19}
							&&\sup_{t\in(t_{0}-{R^{2}},t_{0})}\left\Vert u^{\varepsilon}(\cdot,t)-u(\cdot,t)\right\Vert_{L^{2}\left(B_{R}\left(x_{0}\right)\right)}^{2}+\int_{Q_{R}}\left|H_{\frac{p}{2}}\left(Du^{\varepsilon}\right)-H_{\frac{p}{2}}\left(Du\right)\right|^{2}\,dz\cr\cr
							&\leq&c\left(n,p,{R}\right)\left\Vert f-f^{\varepsilon}\right\Vert_{L^2\left(Q_{R}\right)}^{\frac{n+2}{n+1}}\cdot\left(\int_{Q_{R}}\left(\left|Du\right|^{p}+1\right)\,dz\right)^{\frac{n}{p\left(n+1\right)}}\cr\cr
							&&+c\left(n, p, {R}\right)\left\Vert f-f^{\varepsilon}\right\Vert_{L^2\left(Q_{R}\right)}^\frac{p\left(n+2\right)}{n\left(p-1\right)+p}+\varepsilon\,c_p\int_{Q_{R}}\left|Du\right|^{p}\,dz.
						\end{eqnarray}
						For further needs, we also record that, combining \eqref{eq:ee16} and \eqref{eq:ee19}, we have
						\begin{eqnarray}\label{eq:ee17}
							\int_{Q_{R}}\left|Du^{\varepsilon}\right|^{p}\,dz&\leq&c\left(n,p,{R}\right)\left\Vert f-f^{\varepsilon}\right\Vert_{L^2\left(Q_{R}\right)}^{\frac{n+2}{n+1}}\cdot\left(\int_{Q_{R}}\left(\left|Du\right|^{p}+1\right)\,dz\right)^{\frac{n}{p\left(n+1\right)}}\cr\cr
							&&+c\left(n,p,{R}\right)\left\Vert f-f^{\varepsilon}\right\Vert_{L^2\left(Q_{R}\right)}^\frac{p\left(n+2\right)}{n\left(p-1\right)+p}+\varepsilon\,c_p\int_{Q_{R}}\left|Du\right|^{p}\,dz\cr\cr
							&&+c_p\int_{Q_{R}}\left(\left|Du\right|^{p}+1\right)\,dz.
						\end{eqnarray}
						{\bf Step 2:  The conclusion.}\\
						Let us fix {$\rho>0$ such that $Q_{2\rho}\subset Q_R$}. 
						We start observing that
						\begin{eqnarray*}\label{eq230} &&\int_{Q_{{\frac{\rho}{2}}}}\left|\tau_h\left[\mathcal{G}_\delta\left(\left(\left|Du\right|-\delta -1\right)_+\right)\right]\right|^2\,dz\cr\cr 
							&\le& c \int_{Q_{{\frac{\rho}{2}}}}\left|\tau_h\left[\mathcal{G}_\delta\left(\left(\left|Du^\eps\right|-\delta -1\right)_+\right)\right]\right|^2\,dz\cr\cr &&+c\int_{Q_{{\rho}}}\left|\mathcal{G}_\delta\left(\left(\left|Du^\eps\right|-\delta -1\right)_+\right)-\mathcal{G}_\delta\left(\left(\left|Du\right|-\delta -1\right)_+\right)\right|^2\,dz.
						\end{eqnarray*}
						We estimate the right hand side of previous inequality using \eqref{eq229t} and \eqref{GHineq}, as follows
						\begin{eqnarray*}\label{tGdeltaDu*}
							&&\int_{Q_{{\frac{\rho}{2}}}}\left|\tau_h\left[\mathcal{G}_\delta\left(\left(\left|Du\right|-\delta -1\right)_+\right)\right]\right|^2\,dz\cr\cr 
							&\le&\frac{c|h|^2}{{\rho^2}}\left[\int_{Q_{{2\rho}}}\left(1+\left|D u^\varepsilon\right|^p\right)\,dz+\delta^{2-p}\int_{Q_{{2\rho}}}\left|f^\varepsilon\right|^2 \,dz\right]\cr\cr
							&&+c_p\int_{Q_{{2\rho}}}\left|H_{\frac{p}{2}}\left(Du^\eps\right)-H_{\frac{p}{2}}\left(Du\right)\right|^2\,dz
						\end{eqnarray*}
						that, thanks to \eqref{eq:ee19}, implies
						\begin{eqnarray}\label{tGdeltaDu}
							&&\int_{Q_{{\frac{\rho}{2}}}}\left|\tau_h\left[\mathcal{G}_\delta\left(\left(\left|Du\right|-\delta -1\right)_+\right)\right]\right|^2\,dz\cr\cr 
							&\le&\frac{c|h|^2}{\rho^2}\left[\int_{Q_{{2\rho}}}\left(1+\left|D u^\varepsilon\right|^p\right)\,dz+\delta^{2-p}\int_{Q_{{2\rho}}}\left|f^\varepsilon\right|^2 \,dz\right]\cr\cr
							&&+c\left(n,p,{R}\right)\left\Vert f-f^{\varepsilon}\right\Vert_{L^2\left(Q_{R}\right)}^{\frac{n+2}{n+1}}\cdot\left(\int_{Q_{R}}\left(\left|Du\right|^{p}+1\right)\,dz\right)^{\frac{n}{p\left(n+1\right)}}\cr\cr
							&&+c\left(n, p, {R}\right)\left\Vert f-f^{\varepsilon}\right\Vert_{L^2\left(Q_{R}\right)}^\frac{p\left(n+2\right)}{n\left(p-1\right)+p}+\varepsilon\,c_p\int_{Q_{R}}\left|Du\right|^{p}\,dz.
						\end{eqnarray}
						Now, using \eqref{eq:ee17}, we get
						\begin{eqnarray*}\label{eq:ee17*}
							\int_{Q_{{2\rho}}}\left(1+\left|Du^{\varepsilon}\right|^{p}\right)\,dz
							&\leq&c\left(n, p, {R}\right)\left\Vert f-f^{\varepsilon}\right\Vert_{L^2\left(Q_{{R}}\right)}^{\frac{n+2}{n+1}}\cdot\left(\int_{Q_{{R}}}\left(\left|Du\right|^{p}+1\right)\,dz\right)^{\frac{n}{p\left(n+1\right)}}\cr\cr
							&&+c\left(n,p,{{R}}\right)\left\Vert f-f^{\varepsilon}\right\Vert_{L^2\left(Q_{{R}}\right)}^\frac{p\left(n+2\right)}{n\left(p-1\right)+p}+\varepsilon\,c_p\int_{Q_{{R}}}\left|Du\right|^{p}\,dz\cr\cr
							&&+c_p\int_{Q_{{R}}}\left(\left|Du\right|^{p}+1\right)\,dz
						\end{eqnarray*}
						which, combined with \eqref{tGdeltaDu}, implies
						\begin{eqnarray*}\label{tauhGdeltaDu'}
							&&\int_{Q_{\frac{\rho}{2}}}\left|\tau_h\left[\mathcal{G}_\delta\left(\left(\left|Du\right|-\delta -1\right)_+\right)\right]\right|^2\,dz\cr\cr 
							&\leq&\frac{c\left(n, p\right)|h|^2}{\rho^2}\left[c(R)\left\Vert f-f^{\varepsilon}\right\Vert_{L^2\left(Q_{{R}}\right)}^{\frac{n+2}{n+1}}\cdot\left(\int_{Q_{{R}}}\left(\left|Du\right|^{p}+1\right)\,dz\right)^{\frac{n}{p\left(n+1\right)}}\right.\cr\cr
							&&\left.+c(R)\left\Vert f-f^{\varepsilon}\right\Vert_{L^2\left(Q_{{R}}\right)}^\frac{p\left(n+2\right)}{n\left(p-1\right)+p}+\varepsilon\,\int_{Q_{{R}}}\left|Du\right|^{p}\,dz\right.\cr\cr
							&&\left.+\int_{Q_{{R}}}\left(\left|Du\right|^{p}+1\right)\,dz+\delta^{2-p}\int_{Q_{{R}}}\left|f^\varepsilon\right|^2 \,dz\right].
						\end{eqnarray*}
						Taking the limit as $\varepsilon\to0$, and since $f^\eps\to f$ strongly in $L^2\left(B_R\right)$, we obtain 
						\begin{eqnarray*}\label{tauhGdeltaDu}
							&&\int_{Q_{\frac{\rho}{2}}}\left|\tau_h\left[\mathcal{G}_\delta\left(\left(\left|Du\right|-\delta -1\right)_+\right)\right]\right|^2\,dz\cr\cr 
							&\le& \frac{c\left(n, p\right)|h|^2}{\rho^2}\left[\int_{Q_{{R}}}\left(\left|Du\right|^{p}+1\right)\,dz+\delta^{2-p}\int_{Q_{{R}}}\left|f\right|^2 \,dz\right], 
						\end{eqnarray*}
						and thanks to Lemma \ref{lem:RappIncre}, we have $\mathcal{G}_\delta\left(\left(\left|Du\right|-\delta -1\right)_+\right)\in L^2\left(t_0-{\rho^2},t_0; W^{1,2}\left(B_{\rho}\right)\right)$ with the following estimate
						\begin{eqnarray*}
							&&\int_{Q_\frac{\rho}{2}}\left|D\left[\mathcal{G}_\delta\left(\left(\left|Du\right|-\delta -1\right)_+\right)\right]\right|^2\,dz\cr\cr 
							&\le& \frac{c\left(n, p\right)}{\rho^2}\left[\int_{Q_{ {R}}}\left(\left|Du\right|^{p}+1\right)\,dz+\delta^{2-p}\int_{Q_{R}}\left|f\right|^2 \,dz\right]. 
						\end{eqnarray*}
						Since previous estimate holds true for any $\rho>0$ such that $4\rho<R$, we may choose $\rho=\displaystyle\frac{R}{8}$ thus getting {\eqref{DGdeltaDu}}.}
				\end{proof}
				
				\section{Proof of Theorem \ref{thm2}}
				The higher differentiability result of Theorem \ref{Theorem1}
				allows us to argue as in \cite[Lemma 5.3]{DMS}
				and \cite[Lemma 3.2]{Sche} to obtain the proof of Theorem \ref{thm2}.
				\noindent \begin{proof}[Proof of Theorem \ref{thm2}] We start observing that
					\begin{eqnarray}\label{eq:DV}
						&&\left|D\left(\left(\mathcal{G}_\delta\left(\left(\left|Du^\varepsilon\right|-1-\delta\right)_+\right)\right)^{\frac{4}{np}\,+\,1}\right)\right|\cr\cr 
						&\le& c\left|\mathcal{G}_\delta\left(\left(\left|Du^\varepsilon\right|-1-\delta\right)_+\right)\right|^{\frac{4}{np}}\left|D\left[\mathcal{G}_\delta\left(\left(\left|Du^\varepsilon\right|-1-\delta\right)_+\right)\right]\right|,
					\end{eqnarray}
					\\
					where $c\equiv c(n,p)>0$ and $\mathcal{G}_\delta(t)$ is the function defined at \eqref{Gdef}.\\
					
					With the notation we used in the previous sections, for $B_{{2\rho}}\left(x_{0}\right)\Subset B_R\left(x_0\right)$,
					let $\varphi\in C_{0}^{\infty}\left(B_{\rho}\left(x_{0}\right)\right)$ and $\chi\in W^{1,\infty}\left((0,T)\right)$
					be two non-negative cut-off functions with $\chi(0)=0$ and $\partial_{t}\chi\geq0$.
					Now, we fix a time $t_{0}\in(0,T)$ and apply the Sobolev embedding theorem
					on the time slices $\varSigma_{t}:=B_{\rho}(x_{0})\times\left\{ t\right\} $
					for almost every $t\in(0,t_{0})$, to infer that 
					\begin{eqnarray*}
						&&\int_{\varSigma_{t}}\varphi^{2}\left(\left(\mathcal{G}_\delta\left(\left(\left|Du^\varepsilon\right|-1-\delta\right)_+\right)\right)^{\frac{4}{np}\,+\,1}\right)^{2}\,dx\cr\cr
						&\leq&c\left(\int_{\varSigma_{t}}\left|D\left(\varphi \left(\mathcal{G}_\delta\left(\left(\left|Du^\varepsilon\right|-1-\delta\right)_+\right)\right)^{\frac{4}{np}\,+\,1}\right)\right|^{\frac{2n}{n+2}}\,dx\right)^{\frac{n+2}{n}}\cr\cr
						&\leq&c\left(\int_{\varSigma_{t}}\left|\varphi\,D\left(\left(\mathcal{G}_\delta\left(\left(\left|Du^\varepsilon\right|-1-\delta\right)_+\right)\right)^{\frac{4}{np}\,+\,1}\right)\right|^{\frac{2n}{n+2}}\,dx\right)^{\frac{n+2}{n}}\cr\cr
						&&+c\left(\int_{\varSigma_{t}}\left|\left|\mathcal{G}_\delta\left(\left(\left|Du^\varepsilon\right|-1-\delta\right)_+\right)\right|^{\frac{4}{np}\,+\,1}\,D\varphi\right|^{\frac{2n}{n+2}}\,dx\right)^{\frac{n+2}{n}}\cr\cr
						&=:&\,c\,I_{1}(t)\,+\,c\,I_{2}(t),
					\end{eqnarray*}
					where, in the second to last line, we have applied Minkowski's and
					Young's inequalities one after the other. We estimate
					$I_{1}(t)$ and $I_{2}(t)$ separately. Let us first consider $I_{1}(t)$.
					Using \eqref{eq:DV}, Lemma \ref{estGdelta} and H\"older's inequality with exponents $\displaystyle\left(\frac{n+2}{n}, \frac{n+2}{2}\right)$, we deduce
					\begin{eqnarray*}
						I_{1}(t)&\leq&c\left(\int_{\varSigma_{t}}\varphi^{\frac{2n}{n+2}}\left[\left(\left|Du^\varepsilon\right|-1\right)_{+}^{\frac{2}{n}}\left|D\mathcal{G}_\delta\left(\left(\left|Du^\varepsilon\right|-1-\delta\right)_{+}\right)\right|\right]^{\frac{2n}{n+2}}\,dx\right)^{\frac{n+2}{n}}\cr\cr
						&\leq&c\int_{\varSigma_{t}}\varphi^{2}\left|D\mathcal{G}_\delta\left(\left(\left|Du^\varepsilon\right|-1-\delta\right)_{+}\right)\right|^{2}dx\left(\int_{\supp(\varphi)}\left(\left|Du^\varepsilon\right|-1\right)_{+}^{2}\,dx\right)^{\frac{2}{n}}\cr\cr
						&\leq&c\int_{\varSigma_{t}}\varphi^{2}\left|D\mathcal{G}_\delta\left(\left(\left|Du^\varepsilon\right|-1-\delta\right)_{+}\right)\right|^{2}\,dx\left(\int_{\supp(\varphi)}\left|Du^\varepsilon\right|^{2}\,dx\right)^{\frac{2}{n}}.
					\end{eqnarray*}
					We now turn our attention to $I_{2}(t)$. Lemma \ref{estGdelta} and H\"older's inequality yield
					\begin{eqnarray*}
						I_{2}(t)&\le&c\left(\int_{\varSigma_{t}}\left(\left|Du^\varepsilon\right|-1\right)_{+}^{\frac{np\,+\,4}{n+2}}\left|D\varphi\right|^{\frac{2n}{n+2}}\,dx\right)^{\frac{n+2}{n}}\cr\cr
						&\leq&c\left(\int_{\varSigma_{t}}\left[\left|D\varphi\right|^{2}\left|Du^\varepsilon\right|^{p}\right]^{\frac{n}{n+2}}\left|Du\right|^{\frac{4}{n+2}}\,dx\right)^{\frac{n+2}{n}}\cr\cr
						&\leq&c\int_{\varSigma_{t}}\left|D\varphi\right|^{2}\left|Du^\varepsilon\right|^{p}\,dx\left(\int_{\mathrm{supp}(\varphi)}\left|Du^\varepsilon\right|^{2}\,dx\right)^{\frac{2}{n}}.
					\end{eqnarray*}
					Putting together the last three estimates, using Lemma \ref{estGdelta} in the left hand side, and integrating with respect
					to time, we obtain\\
					\begin{eqnarray}\label{eq:est17}
						\int_{Q^{t_{0}}}\chi\varphi^{2}\,\left(\left|Du^\varepsilon\right|-1\right)_{+}^{p\,+\,\frac{4}{n}}\,dz&\leq&c\int_{0}^{t_{0}}\chi\left[\int_{\supp(\varphi)}\left|Du^\varepsilon(x,t)\right|^{2}dx\right]^{\frac{2}{n}}\cdot\cr\cr
						&&\cdot\left[\int_{\varSigma_{t}}\left(\varphi^{2}\left|D\mathcal{G}_\delta\left(\left(\left|Du^\varepsilon\right|-1-\delta\right)_{+}\right)\right|^{2}+\left|D\varphi\right|^{2}\left|Du\right|^{p}\right)\,dx\right]\,dt\cr\cr
						&\leq&c\int_{Q^{t_{0}}}\chi\left(\varphi^{2}\left|D\mathcal{G}_\delta\left(\left(\left|Du^\varepsilon\right|-1-\delta\right)_{+}\right)\right|^{2}+\left|D\varphi\right|^{2}\left|Du\right|^{p}\right)\,dz\cr\cr
						&&\cdot\left[\sup_{0<t<t_{0},\,\chi(t)\neq0}\int_{\supp(\varphi)}\left|Du^\varepsilon(x,t)\right|^{2}\,dx\right]^{\frac{2}{n}},
					\end{eqnarray}
					where we have used the abbreviation $Q^{t_{0}}:=B_{\rho}\left(x_{0}\right)\times\left(0,t_{0}\right)$.\\
					Now we choose $\chi\in W^{1,\infty}\left((0,T)\right)$ such that
					$\chi\equiv0$ on $\left(0,t_{0}-\rho^{2}\right]$, $\chi\equiv1$ on $\left[t_{0}-\left(\displaystyle\frac{\rho}{2}\right)^{2},T\right)$ and $\partial_{t}\chi\geq0$.
					For $\varphi\in C_{0}^{\infty}\left(B_{\rho}\left(x_{0}\right)\right)$, we
					assume that $\varphi\equiv1$ on $B_{\frac{\rho}{2}}\left(x_{0}\right)$, $0\leq\varphi\leq1$ and $\left|D\varphi\right|\leq\displaystyle\frac{C}{\rho}$.\\
					With these choices \eqref{eq:est17} turns into
					\begin{eqnarray}\label{eq:est18}
						\int_{Q_{\frac{\rho}{2}}}\left(\left|Du^\varepsilon\right|-1\right)_{+}^{p\,+\,\frac{4}{n}}\,dz&\leq&c(n,p)\int_{Q_{\rho}}\left(\left|D\mathcal{G}_\delta\left(\left|Du^\varepsilon\right|-1-\delta\right)_{+})\right|^{2}+\rho^{-2}\left|Du^\varepsilon\right|^{p}\right)dz\cr\cr
						&&\cdot\left[\sup_{t_{0}-\rho^{2}<t<t_{0}}\int_{B_{\rho}\left(x_{0}\right)}\left|Du^\varepsilon(x,t)\right|^{2}dx\right]^{\frac{2}{n}}.
					\end{eqnarray}
					We now use \eqref{uniformest}, in order to estimate the first and second integral on the right-hand
					side of \eqref{eq:est18}, thus getting
					\begin{equation*}
						\int_{Q_{\frac{\rho}{2}}}\left(\left|Du^\varepsilon\right|-1\right)_{+}^{p\,+\,\frac{4}{n}}\,dz\leq\frac{c}{\rho^{\frac{2(n+2)}{n}}}\left[\int_{Q_{{2\rho}}}\left(1+\left|Du^\varepsilon\right|^{p}\right)\,dz+\delta^{2-p}\int_{Q_{{2\rho}}}\left|f^\varepsilon\right|^2\,dz\right]^{\frac{2}{n}+1}.
					\end{equation*}
					Now we use \eqref{eq:ee17} to deduce that
					{\begin{eqnarray}\label{eq:ee19*}
							&&\int_{Q_{\frac{\rho}{2}}}(\vert Du^\varepsilon\vert-1)_{+}^{p\,+\,\frac{4}{n}}\,dz\cr\cr
							&\leq&\frac{c\left(n,p\right)}{\rho^{\frac{2(n+2)}{n}}}\left[c(\rho)\left\Vert f-f^{\varepsilon}\right\Vert_{L^2\left(Q_{{2\rho}}\left(z_0\right)\right)}^{\frac{n+2}{n+1}}\cdot\left(\int_{Q_{{2\rho}}}\left(\left|Du\right|^{p}+1\right)\,dz\right)^{\frac{n}{p\left(n+1\right)}}\right]^{\frac{2}{n}+1}\cr\cr
							&&+\frac{c\left(n,p\right)}{\rho^{\frac{2(n+2)}{n}}}\left[c(\rho)\left\Vert f-f^{\varepsilon}\right\Vert_{L^2\left(Q_{{2\rho}}\right)}^\frac{p\left(n+2\right)}{n\left(p-1\right)+p}+\varepsilon\,c_p\int_{Q_{{2\rho}}}\left|Du\right|^{p}\,dz\right]^{\frac{2}{n}+1}\cr\cr
							&&+\frac{c\left(n,p\right)}{\rho^{\frac{2(n+2)}{n}}}\left[\int_{Q_{{2\rho}}}\left(\left|Du\right|^{p}+1\right)\,dz+\delta^{2-p}\int_{Q_{{2\rho}}}\left|f^\varepsilon\right|^2\,dz\right]^{\frac{2}{n}+1}.
					\end{eqnarray}}
					
					Let us observe that estimate \eqref{eq:ee19} in particular implies that
					{\begin{eqnarray*}
							&&\int_{Q_{{R}}}\left|H_{\frac{p}{2}}\left(Du^{\varepsilon}\right)-H_{\frac{p}{2}}\left(Du\right)\right|^{2}\,dz\cr\cr
							&\leq&c\left(n,p,{R}\right)\left\Vert f-f^{\varepsilon}\right\Vert_{L^2\left(Q_{{R}}\right)}^{\frac{n+2}{n+1}}\cdot\left(\int_{Q_{{R}}}\left(\left|Du\right|^{p}+1\right)\,dz\right)^{\frac{n}{p\left(n+1\right)}}\cr\cr
							&&+c\left(n,p,{R}\right)\left\Vert f-f^{\varepsilon}\right\Vert_{L^2\left(Q_{{R}}\right)}^\frac{p\left(n+2\right)}{n\left(p-1\right)+p}+\varepsilon\,c_p\int_{Q_{{R}}}\left|Du\right|^{p}\,dz.
					\end{eqnarray*}}
					By the strong convergence of $f^\varepsilon \to f$ in $L^2\left(Q_R\right)$, passing to the limit as $\varepsilon \to 0$, from previous estimate we deduce 
					$$\lim_{\varepsilon\to 0}\int_{Q_{{R}}}\left|H_{\frac{p}{2}}\left(Du^{\varepsilon}\right)-H_{\frac{p}{2}}\left(Du\right)\right|^{2}\,dz=0 $$
					that is $H_{\frac{p}{2}}\left(Du^{\varepsilon}\right)\to H_{\frac{p}{2}}\left(Du\right)$, strongly in $L^2\left(Q_{{R}}\right).$ Therefore, up to a not relabelled subsequence, we also have $H_{\frac{p}{2}}\left(Du^{\varepsilon}\right)\to H_{\frac{p}{2}}\left(Du\right)$, a.e. in $Q_{{R}}\left(z_0\right)$ and so
					$$\left(\left|Du^\varepsilon\right|-1\right)_+\to\left(\left|Du\right|-1\right)_+\qquad\mbox{ a.e. in }Q_{{R}}\left(z_0\right)$$
					By Fatou's Lemma, taking the limit as $\varepsilon \to 0$ in both sides of \eqref{eq:ee19*}
					{\begin{eqnarray*}\label{eq:ee19**}
							&&\int_{Q_{\frac{\rho}{2}}}\left(\left|Du\right|-1\right)_{+}^{p\,+\,\frac{4}{n}}\,dz\le\liminf_{\varepsilon\to0}\int_{Q_{\frac{\rho}{2}}}(\vert Du^\varepsilon\vert-1)_{+}^{p\,+\,\frac{4}{n}}\,dz\cr\cr
							&\leq&\frac{c\left(n,p\right)}{\rho^{\frac{2(n+2)}{n}}}\left[\int_{Q_{{2\rho}}}\left(\left|Du\right|^{p}+1\right)\,dz+\delta^{2-p}\int_{Q_{{2\rho}}}\left|f\right|^2\,dz\right]^{\frac{2}{n}+1},
					\end{eqnarray*}}
					which holds for any $\delta\in(0, 1)$, so we can fix $\delta=\displaystyle\frac{1}{2}$, to get the conclusion \eqref{est2}.
				\end{proof}
				
				\printbibliography
			\end{document}